\documentclass[12pt]{amsart}

\usepackage{amsmath}
\usepackage{amstext}
\usepackage{amsfonts}
\usepackage{amssymb, dsfont}
\usepackage{amsthm}
\usepackage{amsrefs}
\usepackage{mathtools}
\usepackage{color}
\usepackage{graphicx}
\usepackage[colorlinks=true, allcolors=blue]{hyperref}
\usepackage{float}
\allowdisplaybreaks
\usepackage{tikz-cd}
\usepackage{amsbsy}
\usepackage{amscd}
\usetikzlibrary{matrix,arrows,decorations.pathmorphing}
\usepackage{enumitem}
\usepackage{setspace}

\usepackage[english]{babel}
\usepackage[utf8x]{inputenc}
\usepackage[T1]{fontenc}

\usepackage[margin=1.0in]{geometry}

\newtheorem{thm}{Theorem}[section]
\newtheorem{lem}[thm]{Lemma}
\newtheorem{cor}[thm]{Corollary}
\newtheorem{prop}[thm]{Proposition}

\theoremstyle{definition}
\newtheorem{example}[thm]{Example}
\theoremstyle{definition}

\theoremstyle{definition}
\newtheorem{defn}[thm]{Definition}
\theoremstyle{definition}
\newtheorem{remark}[thm]{Remark}

\newcommand{\LG}{VN(G)}

\newcommand{\LO}{L^1(G)}
\newcommand{\LOQ}{L^1(\mathbb{G})}

\newcommand{\LOQH}{L^1(\widehat{\mathbb{G}})}

\newcommand{\LOQHP}{L^1(\widehat{\mathbb{G}}')}

\newcommand{\LTQ}{L^2(\mathbb{G})}

\newcommand{\LI}{L^{\infty}(G)}
\newcommand{\LIQ}{L^{\infty}(\mathbb{G})}

\newcommand{\LIQH}{L^{\infty}(\widehat{\mathbb{G}})}

\newcommand{\LIQHP}{L^{\infty}(\widehat{\mathbb{G}}')}

\newcommand{\BH}{\mc{B}(H)}

\newcommand{\LUC}{\mathrm{LUC}(\mathbb{G})}
\newcommand{\RUC}{\mathrm{RUC}(\mathbb{G})}
\newcommand{\BLT}{\mc{B}(L^2(G))}

\newcommand{\BLTQ}{\mc{B}(L^2(\mathbb{G}))}

\newcommand{\McbQl}{M_{cb}^l(L^1(\mathbb{G}))}

\newcommand{\McbQr}{M_{cb}^r(L^1(\mathbb{G}))}

\newcommand{\ten}{\otimes}
\newcommand{\oten}{\overline{\otimes}}
\newcommand{\pten}{\widehat{\otimes}}

\newcommand{\iten}{\otimes^{\vee}}

\newcommand{\vphi}{\varphi}

\newcommand{\G}{\mathbb{G}}

\newcommand{\C}{\mathbb{C}}
\newcommand{\N}{\mathbb{N}}

\newcommand{\R}{\mathbb{R}}

\newcommand{\Amod}{A\hskip2pt\mathbf{mod}}
\newcommand{\LOQmod}{L^1(\mathbb{G})\hskip2pt\mathbf{mod}}
\newcommand{\modA}{\mathbf{mod}\hskip2pt A}
\newcommand{\modLOQ}{\mathbf{mod}\hskip2pt L^1(\mathbb{G})}

\DeclareSymbolFont{lettersA}{U}{txmia}{m}{it}
\DeclareMathSymbol{\W}{\mathord}{lettersA}{151}

\newcommand{\ep}{\varepsilon}

\newcommand{\al}{\alpha}

\newcommand{\lm}{\lambda}

\newcommand{\Gam}{\Gamma}

\newcommand{\Hom}{\mathrm{Hom}}
\newcommand{\Op}{\mathbf{Op}}
\newcommand{\Ban}{\mathbf{Ban}}

\newcommand{\Ad}{\mathrm{Ad}}

\newcommand{\id}{\textnormal{id}}

\providecommand{\norm}[1]{\lVert#1\rVert}

\newcommand{\h}[1]{\widehat{#1}}
\newcommand{\wh}[1]{\widehat{#1}}

\newcommand{\mc}[1]{\mathcal{#1}}
\newcommand{\e}[1]{\emph{#1}}

\newcommand{\la}{\langle}
\newcommand{\ra}{\rangle}
\newcommand{\tr}{\mathrm{tr}}
\newcommand{\rmv}[1]{}

\newcommand{\hs}{\hskip10pt}

\newcommand{\quo}{\twoheadrightarrow}

\begin{document}

\title[Towards a local theory of operator modules]{Finite presentation, the local lifting property, and local approximation properties of operator modules}
\author{Jason Crann}
\email{jasoncrann@cunet.carleton.ca}
\address{School of Mathematics and Statistics, Carleton University, Ottawa, ON, Canada K1S 5B6}


\begin{abstract} We introduce notions of finite presentation and co-exactness which serve as qualitative and quantitative analogues of finite-dimensionality for operator modules over completely contractive Banach algebras. With these notions we begin the development of a local theory of operator modules by introducing analogues of the local lifting property, nuclearity, and semi-discreteness. For a large class of operator modules we prove that the local lifting property is equivalent to flatness, generalizing the operator space result of Kye and Ruan \cite{KR}. We pursue applications to abstract harmonic analysis, where, for a locally compact quantum group $\G$, we show that $\LOQ$-nuclearity of $\mathrm{LUC}(\G)$ and $\LOQ$-semi-discreteness of $\LIQ$ are both equivalent to co-amenability of $\G$. We establish the equivalence between $A(G)$-injectivity of $G\bar{\ltimes}M$, $A(G)$-semi-discreteness of $G\bar{\ltimes} M$, and amenability of $W^*$-dynamical systems $(M,G,\alpha)$ with $M$ injective. We end with remarks on future directions.\end{abstract}

\begin{spacing}{1.0}

\maketitle

\section{Introduction} The local theory of operator spaces, like its Banach space counterpart, aims to study global properties of a space through the structure of its finite-dimensional subspaces. In this context, local properties often translate into homological properties of functors on associated categories. For instance, exactness of a $C^*$-algebra $A$, defined by exactness of the functor $A\iten(\cdot)$, is equivalent to approximate factorizations of an inclusion $A\subseteq\BH$ through full matrix algebras. In fact, some of the deepest theorems in operator algebras concern the equivalence between local and homological characterizations of a certain property. The equivalence of injectivity and approximate finite-dimensionality for von Neumann algebras \cite{Connes} being a prominent example. 

In the setting of operator \textit{modules}, the author's recent work \cite{CN,C,C2,C3} characterizes important properties of a locally compact quantum group $\G$ in terms of homological properties of various operator modules over the convolution algebra $\LOQ$ (or its dual). For instance, $\G$ is amenable if and only if the dual von Neumann algebra $\LIQH$ is injective over $\LOQH$ \cite[Theorem 5.1]{C}. A natural question is whether one can characterize this quantum group injectivity in terms of a ``local property'' inside a relevant module category, analogous to semi-discreteness. Having local characterizations of such homological properties should prove beneficial for the future development of harmonic analysis on locally compact quantum groups.

Motivated by the above (and related) question(s), we introduce a suitable generalization of finite-dimensionality to the context of operator modules through a notion of \textit{(topological) finite presentation}. Analogous to the purely algebraic setting, we define (topological) finite presentation through (topologically) exact sequences of certain free modules. After establishing basic properties and examples, we show that every operator module in $\Amod$ is a direct limit of topologically finitely presented ones.

Through finite presentation we introduce an operator module analogue of the local lifting property. We prove, for a large class of completely contractive Banach algebras $A$, that flatness of an operator module $X\in\Amod$ is equivalent to the local lifting property (see Theorem \ref{t:LLP}). This generalizes the corresponding result for operator spaces by Kye and Ruan \cite[Theorem 5.5]{KR}. Our class includes any $A$ which is the predual of a von Neumann algebra, e.g. $A=\LOQ$. Thus, as a corollary, a locally compact quantum group $\G$ is amenable if and only if $\LOQH$ has the local lifting property in $\LOQH\hskip2pt\mathbf{mod}$ (see Corollary \ref{c:1}), a result which is new even for groups.

We then initiate an investigation into local approximation properties for operator modules, introducing analogues of nuclearity and semi-discreteness. Any nuclear module necessarily has the weak expectation property in the sense of \cite{BC2}, and any semi-discrete module is necessarily injective. Specializing to the setting of locally compact quantum groups, we establish:

\begin{thm}\label{t:intro} Let $\G$ be a locally compact quantum group. The following conditions are equivalent:
\begin{enumerate}
\item $\G$ is co-amenable;
\item $\LIQ$ is semi-discrete in $\modLOQ_1$;
\item $\LIQ$ is semi-discrete in $\LOQmod_1$;
\item $\mathrm{RUC}(\G)=\la\LOQ\star\LIQ\ra$ is nuclear in $\modLOQ_1$;
\item $\mathrm{LUC}(\G)=\la\LIQ\star\LOQ\ra$ is nuclear in $\LOQmod_1$;
\item The inclusion $C_0(\G)\hookrightarrow C_0(\G)^{**}$ is nuclear in $\modLOQ_1$;
\item The inclusion $C_0(\G)\hookrightarrow C_0(\G)^{**}$ is nuclear in $\LOQmod_1$.
\end{enumerate}
\end{thm}

In the setting of $W^*$-dynamical systems $(M,G,\alpha)$ with $M$ is injective (as a von Neumann algebra), we show the equivalence between $A(G)$-injectivity of $G\bar{\ltimes}M$, $A(G)$-semi-discreteness of $G\bar{\ltimes} M$, and amenability of $(M,G,\alpha)$, where $A(G)$ is the Fourier algebra of $G$ acting on $G\bar{\ltimes}M$ through the dual co-action.

Our considerations also lead naturally to an operator module analogue of co-exactness \cite[\S4.5]{Webster}, which measures the completely bounded Banach--Mazur distance from a module to a quotient of a finitely generated ``matrically free'' module of the form $A_+\pten T_n$. For a large class of finitely presented modules $E\in\Amod$, we show that $E$ is $\lm$-co-exact if and only if for any family $(X_i)_{i\in I}$ in $\modA$, and any free ultrafilter $\mc{U}$ on $I$,
$$\bigg(\prod_{i\in I} X_i/\mc{U}\bigg)\pten_A E\cong_\lm\prod_{i\in I} (X_i\pten_A E)/\mc{U},$$
where $\pten_A$ is the projective module tensor product (see Theorem \ref{t:ul}). This is an operator module generalization of Dong's characterization of co-exactness \cite[Theorem 2.2]{Dong2}, which itself is a dual analogue of Pisier's ultraproduct characterization of exactness \cite[Proposition 6]{Pisier}.

The paper is structured as follows. Section \ref{s:prelims} contains the necessary preliminaries on operator spaces, categorical notions and the theory of locally compact quantum groups. Section \ref{s:fp} is devoted to our notion of (topological) finite presentation, its basic properties and examples. Section \ref{s:LLP} introduces the local lifting property (LLP) for modules and establishes the equivalence with flatness. Section \ref{s:nuc} contains the definitions of nuclearity and semi-discreteness, along with basic properties and examples arising from abstract harmonic analysis. Section \ref{s:dyn} contains the proof of the equivalence between $A(G)$-injectivity and $A(G)$-semi-discretenss for crossed products of $W^*$-dynamical systems. Section \ref{s:coexact} concerns our generalization of co-exactness, and the aforementioned ultraproduct characterization.

Several natural lines of investigation are suggested by this work. We therefore finish with concluding remarks on future directions.

\section{Preliminaries}\label{s:prelims}




\subsection{Infinite Matrices and Tensor Products}\label{ss:op}

Given an index set $I$ and an operator space $X$, the space $M_I(X)$ consists of matrices $[x_{i,j}]$ indexed by $i,j\in I$ whose finite submatrices are uniformly bounded, that is,  
$$\sup_{F\subseteq I, |F|<\infty}\{\norm{[x_{i,j}]}\mid i,j\in F\}<\infty.$$
When $I=\{1,...,n\}$, we use the subscript $n$ in lieu of $I$, as usual. Under the canonical identification $M_n(M_I(X))=M_I(M_n(X))$ (by interchanging indicies), $M_I(X)$ becomes an operator space. We write $K_I(X)$ for the closure of the space of finitely supported matrices in $M_I(X)$, i.e., those matrices with only a finite number of nonzero
entries. We write $T_I(X)$ for the subspace of $M_I(X)$ consisting of matrices such that
$$\norm{[x_{i,j}]}_1:=\sup_{F\subseteq I, |F|<\infty}\{\norm{[x_{i,j}]}_1\mid i,j\in F\}<\infty,$$
where $\norm{\cdot}_1$ is the 1-norm on $M_{|F|}(X)$ (see \cite[\S4.1]{ER}).

When $X=\C$, we have 
$$M_I:=M_I(\C)=\mc{B}(\ell^2_I); \ \ K_I:=K_I(\C)=\mc{K}(\ell^2_I); \ \ T_I:=T_I(\C)=\mc{T}(\ell^2_I),$$
that is, we recover the space of bounded, compact, and trace class operators on the Hilbert space $\ell^2_I$, respectively. 

For general operator spaces $X$, we have the following canonical identifications 
\begin{itemize}
\item $M_I(X)=\mc{CB}(T_I,X)$ (see \cite[Proposition 1.5.14 (10)]{BLM});
\item $K_I(X)=X\iten K_I$, where $\iten$ is the injective operator space tensor product (see \cite[Equation (1.37)]{BLM});
\item $T_I(X)= X\pten T_I$, where $\pten$ is the projective operator space tensor product (see \cite[Theorem 10.1.3]{ER} for the case $I=\N$, the general case follows similarly (as mentioned in \cite[\S10.1]{ER})).
\end{itemize}

In addition, the canonical duality relations $K_I(X)^*\cong T_I(X^*)$ and $T_I(X)^*\cong M_I(X^*)$ hold (see \cite[Theorem 10.1.4]{ER} for the case $I=\N$). For further details we refer the reader to \cite[\S10.1]{ER} and \cite[\S1.2.26]{BLM}.

\subsection{Basic Categorical Notions in $\Amod$}\label{ss:catnotions}
Throughout this paper, unless otherwise specified, $A$ is a completely contractive Banach algebra, meaning that $A$ is a Banach algebra for which multiplication extends to a complete contraction $m_A:A\pten A\rightarrow A$. Equivalently,
$$\norm{[a_{i,j}b_{k,l}]}_{mn}\leq\norm{[a_{i,j}]}_{m}\norm{[b_{k,l}]}_n, \ \ \ [a_{i,j}]\in M_m(A), \ [b_{k,l}]\in M_n(A).$$
An operator space $X$ is a left \e{operator $A $-module} if it is a left Banach $ A  $-module such that the module map $m_X:A\pten X\rightarrow X$ is completely contractive. We say that $X$ is \e{faithful} if for every non-zero $x\in X$, there is $a\in A  $ such that $a\cdot x\neq 0$, and we say that $X$ is \e{essential} if $\la A\cdot X \ra=X$, where $\la\cdot\ra$ denotes the closed linear span. We denote by $\Amod$ (respectively, $\Amod_1$) the category of left operator $ A  $-modules with morphisms given by completely bounded (respectively, contractive) module homomorphisms. If $A$ is unital, $X\in\Amod$ is \textit{unital} if $1_A\cdot x=x$ for all $x\in X$. We let $A\hskip2pt\mathbf{unmod}$ denote the resulting subcategory of unital modules in $\Amod$. Given $X,Y$ in $\Amod$, we write the space of morphisms from $X$ to $Y$ as $\Hom(X,Y):=_{A}\mc{CB}(X,Y)$. Right operator $ A  $-modules are defined similarly, and we denote the resulting categories by $\modA$ and $\modA_1$. In $\modA$ we have $\Hom(X,Y)=\mc{CB}_{A}(X,Y)$. We often restrict attention to $\Amod$, the corresponding notions in $\modA$ naturally carrying over. When $A=\C$ we recover the category $\Op$ of operator spaces and completely bounded mappings.  

The operator space dual $X^*$ of any $X\in\Amod$ becomes a right operator $A$-module in the canonical fashion:
\begin{equation}\label{e:dualaction}\la f\cdot a, x\ra=\la f, a\cdot x\ra, \ \ \ f\in X^*, \ a\in A, \ x\in X.\end{equation}
See \cite[Page 309]{ER} for details. $Y\in\modA$ is a \textit{dual (right) operator $ A  $-module} if there exists $X\in\Amod$ for which $Y=X^*$ with the above module structure. Similarly for left modules.

The following terminology will be used throughout the paper. A morphism $\vphi:X\rightarrow Y$ in $\Amod$ is: 
\begin{itemize}
\item a \textit{complete $\lm$-embedding}, often denoted $\hookrightarrow_\lm$, if $\norm{[x_{i,j}]}\leq\lm\norm{[\vphi(x_{i,j})]}$ for all $[x_{i,j}]\in M_n(X)$, $n\in\N$;
\item a \textit{complete $\lm$-quotient} map, often denoted $\twoheadrightarrow_\lm$, if $\vphi_n(M_n(X)_{\norm{\cdot}<\lm})\supseteq M_n(Y)_{\norm{\cdot}<1}$ for every $n\in\N$ (equivalently, the induced map $\widetilde{\vphi}:X/\mathrm{Ker}(\vphi)\rightarrow Y$ is bijective with $\norm{\widetilde{\vphi}^{-1}}_{cb}\leq \lm$);
\item a \textit{$\lm$-strict} morphism if $\vphi(X)$ is closed and the induced map $\widetilde{\vphi}:X/\mathrm{Ker}(\vphi)\rightarrow \vphi(X)$ satisfies $\norm{\widetilde{\vphi}^{-1}}_{cb}\leq\lm$.
\end{itemize}

Let $X\in\modA$ and $Y\in\Amod$. The \e{$ A  $-module tensor product} of $X$ and $Y$ is the quotient space $X\pten_{ A  }Y:=X\pten Y/N$, where
$$N=\la x\cdot a\ten y-x\ten a\cdot y\mid x\in X, \ y\in Y, \ a\in A  \ra,$$
and, again, $\la\cdot\ra$ denotes the closed linear span. It follows that
$$\mc{CB}(X,Y^*)_{ A  }\cong N^{\perp}\cong(X\pten_{ A  } Y)^*\cong\mc{CB}_A(Y,X^*).$$



The identification $ A_+= A  \oplus_1\C$ turns the unitization of $ A  $ into a unital completely contractive Banach algebra, and any $X\in\Amod$ becomes a unital operator $ A_+$-module via the extended action
\begin{equation}\label{e:unit}(a,\lm)\cdot x=a\cdot x +\lm x, \ \ \ a\in A, \ \lm\in\C, \ x\in X.\end{equation}
In fact, there is a canonical categorical equivalence $\Amod\cong A_+\hskip2pt\mathbf{unmod}$. Similarly for right modules. By definition of the extended action (\ref{e:unit}), it follows that $X\pten_A Y= X\pten_{A_+}Y$ for all $X\in\modA$ and $Y\in\Amod$. In particular, when $X=A_+$, we have 
\begin{equation}\label{e:A_+}A_+\pten_A Y=A_+\pten_{A_+} Y\cong Y\end{equation}
via multiplication. Similarly if $Y=A_+$ and $X\in\modA$ is arbitrary.

For any $X\in\Amod $, there is a canonical completely isometric morphism $\Delta_X^+:X\hookrightarrow\mc{CB}( A _+,X)$ given by
\begin{equation*}\Delta_X^+(x)(a)=a\cdot x, \ \ \ x\in X, \ a\in A _+,\end{equation*}
where the left $ A  $-module structure on $\mc{CB}( A _+,X)$ is defined by
\begin{equation*}(a\cdot \vphi)(b)=\vphi(ba), \ \ \ a\in A  , \ \vphi\in\mc{CB}( A _+,X), \ b\in A _+.\end{equation*} 

\begin{remark} If $A$ is a unital completely contractive Banach algebra with norm 1 unit then all concepts and results in this paper carry over naturally inside $A\hskip2pt\mathbf{unmod}$, and one does not require the (unconditional) unitization $A_+$. Since most of our examples are non-unital, unless otherwise stated, we will work inside $\Amod\cong A_+\hskip2pt\mathbf{unmod}$ for general $A$.
\end{remark}

A sequence 
\begin{equation*}
\begin{tikzcd}
X \arrow[r] &Y \arrow[r] &Z
\end{tikzcd}
\end{equation*}
in $\Amod$ is \textit{(topologically) exact} if the image of the first morphism is (norm dense in) the kernel of the second. We require the following lemma concerning topologically exact sequences. The standard proof is included for convenience of the reader.

\begin{lem}\label{l:rightexact} Let $A$ be a completely contractive Banach algebra, and let 
\begin{equation*}
\begin{tikzcd}
X \arrow[r, "\psi"] &Y \arrow[r, two heads, "q"] &Z
\end{tikzcd}
\end{equation*}
be a topologically exact sequence in $\Amod$ with $q$ a complete $\lm$-quotient map for some $\lm\geq 1$. Then for any $W\in\Amod$, the sequence 
\begin{equation}\label{e:l1}
\begin{tikzcd}
\Hom(Z,W) \arrow[r] &\Hom(Y,W) \arrow[r] &\Hom (X,W)
\end{tikzcd}
\end{equation}
is exact in $\Op$, and for any $V\in\modA$, the sequence
\begin{equation}\label{e:l2}
\begin{tikzcd}
V\pten_A X \arrow[r] &V\pten_A Y \arrow[r, two heads] &V\pten_A Z
\end{tikzcd}
\end{equation}
is topologically exact in $\Op$.
\end{lem}

\begin{proof} Let $\widetilde{q}:Y/\overline{\psi(X)}\cong_{cb} Z$ be the induced complete isomorphism. Any morphism 
$$\vphi\in\mathrm{Ker}(\Hom(Y,W)\rightarrow\Hom(X,W))$$ 
satisfies $\vphi\circ\psi=0$. Thus, $\vphi$ induces a morphism $\widetilde{\vphi}\in\Hom(Y/\overline{\psi(X)},W)$. Then $\widetilde{\vphi}\circ\widetilde{q}^{-1}\in\Hom(Z,W)$ and $\widetilde{\vphi}\circ\widetilde{q}^{-1}(q(y))=\vphi(y)$ for all $y\in Y$, meaning $\vphi\in\mathrm{Im}(\Hom(Z,W)\rightarrow\Hom(Y,W))$. It follows that (\ref{e:l1}) is exact. Owing to the fact that a sequence in $\Op$ is topologically exact whenever its dual sequence is topologically w*-exact (by the Hahn-Banach theorem), (\ref{e:l2}) is topologically exact by exactness of (\ref{e:l1}) with $W=V^*$.
\end{proof}

Let $\lm\geq1$. $X\in\Amod$ is \e{$\lm$-projective} if for every $Y,Z\in\Amod$, every complete quotient morphism $q:Y\twoheadrightarrow Z$, every morphism $\vphi:X\rightarrow Z$, and every $\varepsilon>0$, there exists a morphism $\widetilde{\vphi}_\varepsilon:X\rightarrow Y$ such that $\norm{\widetilde{\vphi}_\varepsilon}_{cb}< \lm\norm{\vphi}_{cb}+\varepsilon$  and $q\circ\widetilde{\vphi}_\varepsilon=\vphi$, i.e., the following diagram commutes:

\begin{equation*}
\begin{tikzcd}
                     &Y \arrow[d, two heads, "q"]\\
X \arrow[ru, dotted, "\widetilde{\vphi}_\varepsilon"] \arrow[r, "\vphi"] &Z
\end{tikzcd}
\end{equation*}
For example, $A_+\pten T_n$ is 1-projective for any $n\in\N$, where $T_n$ is the space of $n\times n$ trace class operators. One way to see this is through the following chain of identifications, where the first and last identities use (\ref{e:relfree}) (see below):
\begin{align*}\Hom(A_+\pten T_n,Y/\mathrm{Ker}(q))&\cong\mc{CB}(T_n,Y/\mathrm{Ker}(q))\cong M_n(Y/\mathrm{Ker}(q))\cong M_n(Y)/M_n(\mathrm{Ker}(q))\\
&\cong\Hom(A_+\pten T_n, Y)/\Hom(A_+\pten T_n,\mathrm{Ker}(q)).\end{align*}

\begin{remark} Unlike the purely categorical notion of projectivity, which demands liftings through arbitrary epimorphisms, in this metric setting it is natural to demand liftings through complete quotient maps with norm control of the lifting. Hence the $\ep$ quantifier. Indeed, even in the setting of Banach spaces and quotient maps (a.k.a. metric surjections), without this additional quantifier the ground field $\C$ would fail to be a 1-projective Banach space (see \cite[Footnote 6]{Grothendieck} for the case of $\R$).
\end{remark}


Given $\lm\geq 1$, $X\in\Amod$ is \textit{$\lm$-injective} if for every $Y,Z\in\Amod$, every completely isometric morphism $i:Y\hookrightarrow Z$, and every morphism $\vphi:Y\rightarrow X$, there exists a morphism $\widetilde{\vphi}:Z\rightarrow X$ such that $\norm{\widetilde{\vphi}}_{cb}\leq \lm\norm{\vphi}_{cb}$ and $\widetilde{\vphi}\circ i=\vphi$, that is, the following diagram commutes:

\begin{equation*}
\begin{tikzcd}
Z \arrow[rd, dotted, "\widetilde{\vphi}"]\\
Y \arrow[u, hook, "i"] \arrow[r, "\vphi"] &X
\end{tikzcd}
\end{equation*}

Unlike the case of projectivity, the $\ep$ quantifier is not needed  to ensure a rich source of examples when demanding norm preserving extensions through completely isometric monomorphisms.

A module $X\in\Amod$ is \e{$\lm$-flat} if its dual $X^*$ is $\lm$-injective in $\modA$ with respect to its canonical module structure (\ref{e:dualaction}). It is well-known (see, e.g., \cite[Theorem VII.1.42]{H} or \cite[Proposition 3.5.9]{Wood} for the relative case) that $X\in\Amod$ is $\lm$-flat if and only if for every 1-exact sequence
$$0\rightarrow Y\hookrightarrow Z\twoheadrightarrow Z/Y\rightarrow 0,$$
in $\modA$ the sequence
$$0\rightarrow Y\pten_A X\hookrightarrow_\lm Z\pten_A X\twoheadrightarrow Z/Y\pten_A X\rightarrow 0$$
is exact, and the second arrow is a complete $\lm$-embedding. For example, $A_+\pten T_I$ is 1-flat for any non-empty set $I$, which may be seen through 1-injectivity of $(A_+\pten T_I)^*=\mc{CB}(A_+,M_I)$, which is proved in \cite[Proposition 2.3]{C}.

For a detailed discussion of projective operator spaces (respectively, modules) see \cite{Blech} (respectively, \cite{Hel}). Flatness and injectivity (in the sense above) are treated in detail for Banach modules in \cite{Helem}. See \cite{Wood} for the general relative homological theory of operator modules.


\subsection{Limits}\label{ss:limits} Limits and co-limits in $\Amod$ are defined analogously to the Banach space setting (see \cite[\S I.1]{CLM}, for example). Given a family $(X_i)_{i\in I}$ in $\Amod$, their \textit{product} $\prod_{i\in I} X_i$ is the $\ell^\infty$-direct sum of operator spaces (see, e.g., \cite[1.2.17]{BLM}) with diagonal module structure:
$$a\cdot(x_i)=(a\cdot x_i), \ \ \  (x_i)\in\prod_{i\in I} X_i, \ a\in A.$$
It is defined uniquely through the following universal property: given $Y\in\Amod$ and a uniformly completely bounded family of morphisms $\vphi_i:Y\rightarrow X_i$, with $\norm{\vphi_i}_{cb}\leq \lm$, there is a unique morphism $\vphi:Y\rightarrow \prod_{i\in I} X_i$ with $\norm{\vphi_i}_{cb}\leq \lm$ such that $\pi_i\circ\vphi=\vphi_i$ for all $i$, where $\pi_i$ is the canonical projection onto the $X_i$ factor. 

The \textit{coproduct} $\bigoplus_{i \in I} X_i$ of $(X_i)_{i \in I}$ is the $\ell^1$-direct sum of operator spaces (see, e.g., \cite[1.4.13]{BLM}) with diagonal module structure:
$$a\cdot(x_i)=(a\cdot x_i), \ \ \ (x_i)\in\bigoplus_{i \in I} X_i, \ a\in A.$$
It is defined uniquely through the following universal property: given $Y\in\Amod$ and a uniformly completely bounded family of morphisms $\vphi_i:X_i\rightarrow Y$ with $\norm{\vphi_i}_{cb}\leq\lm$ there is a unique morphism $\vphi:\sum_{i \in I} X_i\rightarrow Y$ such that $\norm{\vphi}_{cb}\leq \lm$ and $\vphi\circ \iota_i=\vphi_i$ for all $i$, where $\iota_i$ is the canonical inclusion of the $X_i$ factor. 

An inductive system $(X_i,\vphi_{j,i})_{i,j\in I}$ in $\Amod_1$ is a directed family $(X_i)_{i\in I}$ of modules together with a family of morphisms $\vphi_{j,i}:X_i\rightarrow X_j$, for $i\leq j$, satisfying $\vphi_{i,i}=\id_{X_i}$ and $\vphi_{k,j}\circ\vphi_{j,i}=\vphi_{k,i}$ whenever $k\geq j\geq i$. An operator module $X\in\Amod$ together with a family of morphisms $\vphi_i:X_i\rightarrow X$ in $\Amod_1$ is a \textit{direct limit} of the inductive system $((X_i),\vphi_{j,i})$, if the universal property, illustrated by
\begin{equation}\label{d:co-limit}
\begin{tikzcd}
X_{i} \arrow[dd, "\vphi_{j,i}"]\arrow[rd, "\vphi_i"]\arrow[rdrr, "\psi_i"]\\
& X\arrow[rr, dotted, "\exists! \psi"] & &Y\\
X_{j}\arrow[ru, "\vphi_j"]\arrow[rurr, "\psi_j"]
\end{tikzcd}
\end{equation}
is satisfied. That is, for every $Y\in\Amod$ and uniformly completely bounded family of morphisms $\psi_i:X_i\rightarrow Y$ with $\norm{\psi_i}_{cb}\leq\lm$ making the above diagrams commute, there exists a unique morphism $\psi:X\rightarrow Y$ such that $\norm{\psi}_{cb}\leq\lm$ and $\psi\circ\vphi_i=\psi_i$ for each $i$. We denote $X$ by $\varinjlim_i X_i$. Analogous to the Banach space setting \cite{CLM}, it coincides with the quotient of $\bigoplus_{i \in I}X_i$ by the closed submodule generated by $\vphi_i(x_i)-\vphi_{j}(\vphi_{j,i}(x_i))$, for all $x_i\in X_i$, and $i\leq j$. We can also identify $\varinjlim_i X_i$ with the subspace $\overline{\pi_\infty(X_\infty)}$ of the asymptotic product $\prod_{i\in I} X_i/\sum_{i\in I} X_i$, where $\sum_{i\in I}X_i$ is the $c_0$-direct sum,
$$\pi_\infty:\prod_{i\in I} X_i\twoheadrightarrow\prod_{i\in I} X_i/\sum_{i\in I} X_i$$
is the canonical quotient map, and $X_\infty$ is the submodule 
$$\{(x_i)\in\prod_i X_i\mid \exists \ i_0 \ \textnormal{with} \ x_j=\vphi_{j,i}(x_i), \ j\geq i\geq i_0\}.$$
For sequential inductive systems in $\Op$ with completely isometric connecting maps, this fact is proven in \cite[\S1.3]{ER}. Although well-known in the $C^*$-context, we include a proof for completeness.

\begin{prop}\label{p:limitrep} Let $(X_i,\vphi_{j,i})_{i,j\in I}$ be an inductive system in $\Amod_1$. Then $\varinjlim_i X_i\cong\overline{\pi_\infty(X_\infty)}$, completely isometrically.
\end{prop}

\begin{proof} The standard argument shows that
$$\norm{(\pi_\infty)_n(x_i)}=\limsup_i\norm{x_i}, \ \ \ (x_i)\in M_n\bigg(\prod_{i\in I} X_i\bigg), \ n\in\N.$$
(See \cite[Proposition A.6.1]{ER} for the case of sequences of Banach spaces.)

For each $i\in I$, let $\vphi_i:X_i\rightarrow\overline{\pi_\infty(X_\infty)}$ be given by $\pi_\infty((\vphi_i(x_i)_j))$, where $\vphi_i(x_i)_j=\vphi_{j,i}(x_i)$, if $j\geq i$ and $\vphi_i(x_i)_j=0$ otherwise. Clearly, $(\vphi_i(x_i)_j)\in X_\infty$ as for every $k\geq j\geq i$, $\vphi_{k,j}(\vphi_i(x_i)_j)=\vphi_{k,j}(\vphi_{j,i}(x_i))=\vphi_{k,i}(x_i)=\vphi_i(x_i)_k$. Also,
$$\norm{\vphi_i(x_i)}=\norm{\pi_\infty((\vphi_i(x_i)_j))}=\limsup_j\norm{\vphi_i(x_i)_j}=\limsup_j\norm{\vphi_{j,i}(x_i)}\leq\norm{x_i}.$$
The same argument shows that $\vphi_i$ is completely contractive, and thus a morphism in $\Amod_1$.

We show that $\overline{\pi_\infty(X_\infty)}$ together with the family $(\vphi_i)$ satisfies the universal property of the direct limit. To that end, let $Y\in\Amod$ and $\psi_i:X_i\rightarrow Y$ be a family of morphisms in $\Amod_1$ for which $\norm{\psi_i}_{cb}\leq \lm$ for all $i$ and $\psi_j\circ\vphi_{j,i}=\psi_i$ for all $i\leq j$. First define $\psi:X_\infty\rightarrow Y$ by $\psi((x_i))=\psi_{i_0}(x_{i_0})$. This definition is independent of the $i_0$ chosen: if $j_0\in I$ such that $x_j=\vphi_{j,i}(x_i)$ for all $j\geq i\geq j_0$, pick $k_0\geq i_0,j_0$. Then
$$\psi_{i_0}(x_{i_0})=\psi_{k_0}(\vphi_{k_0,i_0}(x_{i_0}))=\psi_{k_0}(x_{k_0})=\psi_{k_0}(\vphi_{k_0,j_0}(x_{j_0}))=\psi_{j_0}(x_{j_0}).$$
Also, if $x=(x_i)$ and $y=(y_i)$ belong to $X_\infty$ and $\pi_\infty(x)=\pi_\infty(y)$, then for every $\ep>0$, there exists $i_\ep$ such that $\norm{x_i-y_i}<\ep$ for $i\geq i_\ep$. Pick $i_0\geq i_\ep$ for which $\vphi_{j,i}(x_i)=x_j$ and $\vphi_{j,i}(y_i)=y_j$ for all $j\geq i\geq i_0$. Then
$$\norm{\psi(x)-\psi(y)}=\norm{\psi_{i_0}(x_{i_0})-\psi_{i_0}(y_{i_0})}\leq\lm\norm{x_{i_0}-y_{i_0}}<\ep.$$
Hence, $\psi(x)=\psi(y)$, so $\psi$ induces a well defined $A$-module map $\widetilde{\psi}:\pi_\infty(X_\infty)\rightarrow Y$. Finally,
$$\norm{\widetilde{\psi}(\pi_\infty((x_i)))}=\norm{\psi_{i_0}(x_{i_0})}\leq\lm\sup_{i\geq i_0}\norm{x_i},$$
for any $i_0$ satisfying $x_j=\vphi_{j,i}(x_i)$, $j\geq i\geq i_0$. Hence,
$$\norm{\widetilde{\psi}(\pi_\infty((x_i)))}\leq\lm\limsup_i\norm{x_i}=\lm\norm{\pi_\infty((x_i))}.$$
A similar argument shows that $\norm{\widetilde{\psi}}_{cb}\leq\lm$. Hence $\widetilde{\psi}$ extends to a morphism $\overline{\pi_\infty(X_\infty)}\rightarrow Y$ in $\Amod$ satisfying
$$\widetilde{\psi}(\vphi_i(x_i))=\psi_j(\vphi_i(x_i)_j)=\psi_j(\vphi_{j,i}(x_i))=\psi_i(x_i).$$
Uniqueness is obvious.
\end{proof}

We also require the following standard result. The categorical Banach space argument from \cite[1.8e]{CLM} generalizes verbatim as the universal property of $\pten_A$ implies that
$$\mc{CB}(Y\pten_A X,Z)\ni \vphi\mapsto (x\mapsto (y\mapsto\vphi(y\ten_A x))\in\Hom(X,\mc{CB}(Y,Z)),$$ 
is a completely isometric isomorphism for any $X\in\Amod_1$, $Y\in\modA_1$ and $Z\in\Op_1$ (see \cite[Proposition 3.5.9]{BLM}).

\begin{prop}\label{p:iso} Let $(X_i,\vphi_{j,i})_{i,j\in I}$ be an inductive system in $\Amod_1$. Then for any $Y\in\modA_1$, $Y\pten_A(\varinjlim_i X_i)\cong \varinjlim_i(Y\pten_A X_i)$ completely isometrically.
\end{prop}

\subsection{Remarks on Freeness}\label{ss:free} Let $X\in\Op$, and let $\mc{F}:\Amod\rightarrow\mathbf{Op}$ denote the forgetful functor. For $Y$ in $\Amod$, the map 
\begin{equation}\label{e:relfree}\Hom( A_+\pten X,Y)\ni\vphi\mapsto\vphi|_{1\ten X}\in\mc{CB}(X,\mc{F}(Y))\end{equation}
is a complete contraction with completely contractive inverse 
$$\mc{CB}(X,\mc{F}(Y))\ni\psi\mapsto m^+_{Y}\circ (\id_{A^+}\ten\psi)\in \Hom( A_+\pten X,Y).$$
Thus, $\Hom( A_+\pten X,Y)\cong\mc{CB}(X,\mc{F}(Y))$ completely isometrically in $\mathbf{Op}$. 

This isomorphism can be interpreted as the ``relative freeness'' of $A_+\pten X$ in $\Amod$. Indeed, it says that any \textit{linear} completely bounded map $X\rightarrow Y$ admits a unique $A$-module extension to a morphism $A_+\pten X\rightarrow Y$. The relativity in this description is that between the operator space and $A$-module structures of $X$ and $Y$. There is a ``strict'' notion of freeness in $\Amod$ which simultaneously incorporates the operator space and $A$-module structures. This was defined by Helemskii \cite{Hel} following the categorical approach to free objects (see \cite[Definition 2.10]{Hel}, \cite[Chapter III]{MacLane}). In the case of $\Op$, free objects are defined with respect to the functor
\begin{equation}\label{e:functor} V\ni\Op\mapsto \prod_{n\in\N}\mathrm{Ball}(M_n(V))\in\mathbf{Set};\end{equation}
the infinite product in (\ref{e:functor}) corresponding to the completely bounded nature of morphisms in $\Op$. Note that free operator spaces in this sense are always infinite dimensional, e.g., the free operator space over the singleton $I=\{i_0\}$ is $\bigoplus_{n=1}^\infty T_n$ (see \cite[Theorem 5.9]{Hel}). In particular, finite-dimensional trace class operators $T_n$ are not free. Also, for any infinite set $I$, $T_I$ is not free in $\Op$, as it is not projective (see \cite[Corollary 3.11]{Blech}, \cite{ER2}), and any free object in $\Amod$ is projective \cite[Proposition 2.11]{Hel}. Nevertheless, the family of operator spaces $T_I$ satisfies a variant of freeness, which is a matricial analogue of free Banach spaces (see Remark \ref{r:free} below). 

Given a non-empty set $I$, we call a function $f:I\times I\rightarrow X$ \textit{matricially contractive} if $\norm{[f(i,j)]}_{M_I(X)}\leq 1$.

\begin{prop}\label{p:free} Let $A$ be a completely contractive Banach algebra, let $I$ be a non-empty set, and let $\{e_{i,j}\}$ be the standard matrix units in $T_I$. Then $f:I\times I\ni (i,j)\mapsto 1\ten e_{i,j}\in A_+\pten T_I$ is a matricial contraction, and the module $A_+\pten T_I$ satisfies the following universal property in $\Amod_1$ with respect to $I$ and $f$:
\begin{enumerate}
\item $A_+\pten T_I=\overline{\mathrm{span}_{A_+}}\{f(i,j)\mid i,j\in I\}$; and
\item for any $Y\in\Amod$ and any matrically contractive $g:I\times I\rightarrow Y$, there exists a unique morphism $\vphi:X\rightarrow Y$ in $\Amod_1$ such that $\vphi\circ f=g$. 
\end{enumerate}
\end{prop}

\begin{proof} First, we show that the map 
$$f:I\times I \ni (i,j)\mapsto 1\ten e_{i,j}\in A_+\pten T_I$$
is a matricial contraction. Since $T_I\ni\rho\mapsto 1\ten\rho\in A_+\pten T_I$ is a complete isometry, and $M_I(T_I)=M_I(K_I^*)=\mc{CB}(K_I,M_I)$ \cite[Proposition 1.2.29]{BLM}, we have
\begin{align*}\norm{[1\ten e_{i,j}]}_{M_I(A_+\pten T_I)}&=\norm{[e_{i,j}]}_{M_I(T_I)}\\
&=\sup\{\norm{[\la e_{i,j},A_{k,l}\ra]}_{M_n(M_I)}\mid [A_{k,l}]\in\mathrm{Ball}(M_n(K_I)), \ n\in\N\}\\
&=\sup\{\norm{[(A_{k,l})_{i,j}]}_{M_n(M_I)}\mid [A_{k,l}]\in\mathrm{Ball}(M_n(K_I)), \ n\in\N)\}\\
&=1,\end{align*}
where the third equality uses the canonical scalar pairing
$$\la\cdot,\cdot\ra:T_I\times K_I\ni (\rho,x)=\sum_{i,j\in I} \rho_{i,j}x_{i,j}\in\C.$$
Next, given $Y\in\Amod$ and a matricial contraction $g:I\times I\rightarrow Y$, define 
\begin{equation}\label{e:Th}\vphi_g:T_I(A_+)\ni [a_{i,j}]\mapsto \sum_{i,j\in I} a_{i,j}\cdot g(i,j)\in Y.\end{equation}
This expression is well defined as $\vphi_g([a_{i,j}])=\vphi([a_{i,j}]\ten [g(i,j)])$, where $\vphi:T_I(A_+)\pten M_I(Y)\rightarrow Y$ is the complete contraction given by the composition
\begin{align*}T_I(A_+)\pten M_I(Y)&=(A_+\pten T_I)\pten\mc{CB}(T_I,Y)\\
&=(A_+\pten T_I)\pten\Hom(A_+\pten T_I,Y)\\
&\rightarrow Y,
\end{align*}
where the first equality uses the identifications $T_I(A_+)= A_+\pten T_I$, and $M_I(Y)=\mc{CB}(T_I,Y)$ (see section \ref{ss:op}), the second equality uses (\ref{e:relfree}), and the arrow is evaluation (using the universal property of $\pten$ \cite[Proposition 7.1.2]{ER}). Thus, $\vphi_g$ is a completely contractive morphism satisfying $\vphi_g(1\ten e_{i,j})=g(i,j)$, $i,j\in I$. 
\end{proof}

\begin{remark}\label{r:free}
\begin{enumerate}
\item It follows that $A_+\pten T_I$ is the unique module in $\Amod_1$ (up to isomorphism) satisfying properties (1)-(2). We will refer to these properties as the \textit{matricial freeness} $A_+\pten T_I$. In particular, $T_I$ is a ``matricially free'' operator space over $I$.
\item Free objects in the category $\Ban_1$ of Banach spaces and contractive linear maps are of the form $\ell^1(I)$ for a non-empty set $I$ (see \cite[1.11, Remark 2]{CLM}). Indeed, $f:I\ni i\mapsto e_i\in\mathrm{Ball}(\ell^1(I))$ is a function such that for any $Y\in\Ban_1$ and function $g:I\rightarrow\mathrm{Ball}(Y)$, there exists a unique morphism $\vphi:X\rightarrow Y$ satisfying $\vphi\circ f=g$. Viewing $g\in\mathrm{Ball}(\ell^\infty(I,Y))$, when $Y$ is an operator space, the natural matricial analogue of $\mathrm{Ball}(\ell^\infty(I,Y))$ is the unit ball of $M_I(Y)$. Thus, the matricial freeness of $T_I$ is the operator space analogue of the freeness of $\ell^1(I)$.
\end{enumerate}
\end{remark}

The duality $\mc{CB}(A_+,M_I)=(A_+\pten T_I)^*$ motivates the next definition, which will underlie the approximation properties in section \ref{s:nuc}.

\begin{defn}\label{d:cofree} Let $A$ be a completely contractive Banach algebra. A module $X\in\modA$ is \textit{co-free} if it is of the form $\mc{CB}(A_+,M_I)$ for some non-empty set $I$. When $|I|<\infty$ we say that $X$ is \textit{finitely co-generated}. Similar definitions apply to left modules.
\end{defn}

Any co-free module is necessarily 1-injective. See the proof of \cite[Proposition 2.3]{C} (which is modelled off of \cite[Lemma 1]{Ham78}).

\subsection{Locally Compact Quantum Groups}\label{ss:qg} A \textit{locally compact quantum group} is a quadruple $\G=(\LIQ,\Gam,h_L,h_R)$, where $\LIQ$ is a Hopf-von Neumann algebra with co-product $\Gam:\LIQ\rightarrow\LIQ\oten\LIQ$, and $h_L$ and $h_R$ are fixed left and right Haar weights on $\LIQ$, respectively \cite{KV1,KV2}.

For every locally compact quantum group $\G$, there exists a \e{left fundamental unitary operator} $W$ on $L^2(\G,h_L)\ten L^2(\G,h_L)$ and a \e{right fundamental unitary operator} $V$ on $L^2(\G,h_R)\ten L^2(\G,h_R)$ implementing the co-product $\Gam$ via
\begin{equation*}\Gam(x)=W^*(1\ten x)W=V(x\ten 1)V^*, \ \ \ x\in\LIQ.\end{equation*}
Both unitaries satisfy the \e{pentagonal relation}; that is,
\begin{equation}\label{e:pentagonal}W_{12}W_{13}W_{23}=W_{23}W_{12}\hs\hs\mathrm{and}\hs\hs V_{12}V_{13}V_{23}=V_{23}V_{12}.\end{equation}
By \cite[Proposition 2.11]{KV2}, we may identify $L^2(\G,h_L)$ and $L^2(\G,h_R)$, so we will simply use $\LTQ$ for this Hilbert space throughout the paper. 

Let $\LOQ$ denote the predual of $\LIQ$. Then the pre-adjoint of $\Gam$ induces an associative completely contractive multiplication on $\LOQ$, defined by
\begin{equation*}\star:\LOQ\pten\LOQ\ni f\ten g\mapsto f\star g=\Gam_*(f\ten g)\in\LOQ.\end{equation*}
The canonical $\LOQ$-bimodule structure on $\LIQ$ is given by
\begin{equation*}f\star x=(\id\ten f)\Gam(x)\hs\hs\mathrm{and}\hs\hs x\star f=(f\ten\id)\Gam(x)\end{equation*}
for $x\in\LIQ$, and $f\in\LOQ$. A \e{left invariant mean on $\LIQ$} is a state $m\in \LIQ^*$ satisfying
\begin{equation*}\la m,x\star f \ra=\la f,1\ra\la m,x\ra, \ \ \ x\in\LIQ, \ f\in\LOQ.\end{equation*}
Right and two-sided invariant means are defined similarly. A locally compact quantum group $\G$ is \e{amenable} if there exists a left invariant mean on $\LIQ$. It is known that $\G$ is amenable if and only if there exists a right (equivalently, two-sided) invariant mean. $\G$ is \e{co-amenable} if $\LOQ$ has a bounded left (equivalently, right or two-sided) approximate identity \cite[Theorem 3.1]{BT}. 

For general $\G$, the \e{left regular representation} $\lm:\LOQ\rightarrow\BLTQ$ is defined by
\begin{equation*}\lm(f)=(f\ten\id)(W), \ \ \ f\in\LOQ,\end{equation*}
and is an injective, completely contractive homomorphism from $\LOQ$ into $\BLTQ$. Then $\LIQH:=\{\lm(f) : f\in\LOQ\}''$ is the von Neumann algebra associated with the dual quantum group $\h{\G}$. Analogously, we have the \e{right regular representation} $\rho:\LOQ\rightarrow\BLTQ$ defined by
\begin{equation*}\rho(f)=(\id\ten f)(V), \ \ \ f\in\LOQ,\end{equation*}
which is also an injective, completely contractive homomorphism from $\LOQ$ into $\BLTQ$. Then $\LIQHP:=\{\rho(f) : f\in\LOQ\}''$ is the von Neumann algebra associated to the quantum group $\h{\G}'$, and satisfies $\LIQHP=\LIQH'$.

If $G$ is a locally compact group, we let $\G_a=(\LI,\Gam_a,h_L,h_R)$ denote the \e{commutative} quantum group associated with the commutative von Neumann algebra $\LI$, where the co-product is given by $\Gam_a(f)(s,t)=f(st)$, and $h_L$ and $h_R$ are integration with respect to a left and right Haar measure, respectively. The fundamental unitaries in this case are given by
$$W_a\xi(s,t)=\xi(s,s^{-1}t), \ \ V_a\xi(s,t)=\xi(st,t)\Delta(t)^{1/2}, \ \ \ \xi\in L^2(G\times G).$$
The dual $\h{\G}_a$ of $\G_a$ is the \e{co-commutative} quantum group $\G_s=(\LG,\Gam_s,h)$, where $\LG$ is the left group von Neumann algebra with co-product $\Gam_s(\lm(t))=\lm(t)\ten\lm(t)$, and $h$ is the Plancherel weight. Then $L^1(\G_a)$ is the usual group convolution algebra $\LO$, and $L^1(\G_s)$ is the Fourier algebra $A(G)$.


For general $\G$, we let $C_0(\G):=\overline{\hat{\lm}(\LOQH)}^{\norm{\cdot}}$ denote the \e{reduced quantum group $C^*$-algebra} of $\G$. Here,
$$\hat{\lm}(\hat{f})=(\hat{f}\ten\id)(\wh{W}), \ \ \ \hat{f}\in\LOQH,$$
where $\wh{W}=\sigma W^*\sigma$ is the left fundamental unitary of $\wh{\G}$, and $\sigma$ is the flip map on $\LTQ\ten\LTQ$. The operator dual $M(\G):=C_0(\G)^*$ is a completely contractive Banach algebra containing $\LOQ$ as a norm closed two-sided ideal via the map $\LOQ\ni f\mapsto f|_{C_0(\G)}\in M(\G)$. The co-product satisfies $\Gam(C_0(\G))\subseteq M(C_0(\G)\iten C_0(\G))$, where $M(C_0(\G)\ten^{\vee} C_0(\G))$ is the multiplier algebra of the $C^*$-algebra $C_0(\G)\iten C_0(\G)$, and $\iten$ is the injective operator space tensor product (which coincides with the minimal $C^*$-tensor norm on $C^*$-algebras).

We let $C_0^u(\G)$ be the \e{universal quantum group $C^*$-algebra} of $\G$ \cite{K}, and denote the canonical surjective $*$-homomorphism onto $C_0(\G)$ by $\Lambda_{\G}:C_0^u(\G)\rightarrow C_0(\G)$. The space $C_0^u(\G)^*$ then has the structure of a unital completely contractive Banach algebra such that the map $\LOQ\rightarrow C_0^u(\G)^*$ given by the composition of the inclusion $\LOQ\subseteq M(\G)$ and $\Lambda_{\G}^*:M(\G)\rightarrow C_0^u(\G)^*$ is a completely isometric homomorphism. It follows that $\LOQ$ is a norm closed two-sided ideal in $C_0^u(\G)^*$ \cite[Proposition 8.3]{K}. 

An element $\hat{b}'\in\LIQHP$ is a \e{completely bounded right multiplier} of $\LOQ$ if $\rho(f)\hat{b}'\in\rho(\LOQ)$ for all $f\in\LOQ$ and the induced map
\begin{equation*}m_{\hat{b}}^r:\LOQ\ni f\mapsto\rho^{-1}(\rho(f)\hat{b}')\in\LOQ\end{equation*}
is completely bounded on $\LOQ$. We let $\McbQr$ denote the space of completely bounded right multipliers of $\LOQ$, which is a completely contractive Banach algebra with respect to the norm
\begin{equation*}\norm{[\hat{b}_{ij}]}_{M_n(\McbQr)}=\norm{[m^r_{\hat{b}_{ij}}]}_{cb}.\end{equation*}
Completely bounded left multipliers are defined analogously and we denote by $\McbQl$ the corresponding completely contractive Banach algebra. 

Given $\hat{b}'\in\McbQr$, the adjoint $\Theta^r(\hat{b}'):=(m_{\hat{b}}^r)^*$ defines a normal completely bounded right $\LOQ$-module map on $\LIQ$. When $\hat{b}'=\rho(f)$, for some $f\in\LOQ$, the map $\Theta^r(\rho(f))$ is nothing but the convolution action of $\LOQ$ on $\LIQ$, that is,
$$\Theta^r(\rho(f))(x)=f\star x=(\id\ten f)\Gam(x), \ \ \ x\in\LIQ.$$
Moreover, the map
\begin{equation}\label{e:Theta}\Theta^r:\McbQr\cong\mc{CB}^{\sigma}_{\LOQ}(\LIQ)=\mc{CB}_{\LOQ}(C_0(\G),\LIQ)\end{equation}
induces a completely isometric isomorphism of completely contractive Banach algebras \cite{JNR}. The second equality follows from \cite[Proposition 4.1]{JNR}. Since $\LOQ$ is a right ideal in $C_0^u(\G)^*$, any $\mu\in C_0^u(\G)^*$ defines a completely bounded left $\LOQ$-module map on $\LOQ$ by right multiplication, and therefore by duality, a completely bounded multiplier in $\mc{CB}^{\sigma}_{\LOQ}(\LIQ)$. The resulting map will be denoted $\Theta^r(\mu)$ for simplicity.

\section{Finitely Presented Modules}\label{s:fp}

A right module $M$ over a unital ring $R$ is finitely presented if and only if there is an exact sequence
\begin{equation*}
\begin{tikzcd}
R^m \arrow[r] &R^n \arrow[r] &M \arrow[r] &0.
\end{tikzcd}
\end{equation*}
This algebraic formulation motivates the next definition. 

\begin{defn}\label{d:fp} Let $A$ be a completely contractive Banach algebra. An object $E$ in $\Amod_1$ is \textit{finitely presented} (respectively, \textit{topologically finitely presented}) if there is an exact (respectively, topologically exact) sequence in $\Amod_1$
\begin{equation}\label{e:fp1}
\begin{tikzcd}
A_+\pten F_1 \arrow[r] &A_+\pten F_2 \arrow[r, two heads, "\vphi"] &E,
\end{tikzcd}
\end{equation}
where $F_1$ and $F_2$ are finite-dimensional operator spaces and $\vphi$ is a complete quotient morphism.
\end{defn} 

\begin{remark}\label{r:1}
\begin{enumerate}
\item The (topological) exactness of (\ref{e:fp1}) in $\Amod_1$ means that all arrows are completely contractive $A$-module maps, and that the image of the first arrow is (norm dense in) the kernel of the second. 
\item Clearly, finite presentation implies topological finite presentation. Note that the topological adjective in the latter refers to it ``relations'' and not its algebraic generation. A more accurate choice of terminology would be finitely generated and topologically finitely related. However, for simplicity, we stick with the chosen terminology.
\item One is tempted to use exact sequences of the form
\begin{equation*}
\begin{tikzcd}
A_+\pten F_1 \arrow[r] & A_+\pten F_2 \arrow[r] &E\arrow[r] & 0
\end{tikzcd}
\end{equation*}
in Definition \ref{d:fp}. However, the quotient property of $\vphi$ will be important in what follows, and a surjective morphism need not be a complete quotient map, so this property would still have to be specified. We find the concise sequence (\ref{e:fp1}) appropriate for our needs.
\item We have defined finite presentation through an exact sequence of ``relatively free'' modules (see section \ref{ss:free}). Based on the purely algebraic definition, a perhaps more natural definition would make use of a stricter notion of freeness which incorporates both the operator space and $A$-module structure. As we will show in Proposition \ref{p:refp}, any (topologically) finitely presented $E\in\Amod_1$ can be written through a (topologically) exact sequence
\begin{equation*}
\begin{tikzcd}
A_+\pten T_m \arrow[r] & A_+\pten T_n \arrow[r, two heads, "\vphi"] &E
\end{tikzcd}
\end{equation*}
with $\vphi$ a complete $\lm$-quotient morphism for some $\lm\geq 1$. (It is not always possible to take $\lm=1$; see section \ref{s:coexact} for a detailed analysis.) Hence, any finitely presented module is derived through an exact sequence of finitely generated matricially free modules.
\end{enumerate}
\end{remark}

We now present elementary examples, fixing a completely contractive Banach algebra $A$.

\begin{example}\label{ex:fp1} Let $F$ be a finite-dimensional operator space. Then $A_+\pten F$ is finitely presented. In particular, when $A=\C$, any finite-dimensional operator space $F$ is finitely presented. Since the converse is obvious, an operator space is finitely presented if and only if it is finite dimensional. 
\end{example}

\begin{example}\label{ex:fp2} Let $P$ be $\lm$-projective in $\Amod$ and finitely generated in the sense that there is a complete quotient morphism $q:A_+\pten F\twoheadrightarrow P$ for some finite-dimensional operator space $F$. Then by projectivity, for every $\ep>0$ there exists $\vphi:P\rightarrow A_+\pten F$ for which $q\circ\vphi=\id_P$ and $\norm{\vphi}_{cb}<\lm+\ep$. The morphism $\vphi\circ q:A_+\pten F\rightarrow A_+\pten F$ is a projection onto $\vphi(P)$, and $\psi:=\id_{A_+\pten F}-\vphi\circ q$ is a projection onto $\mathrm{Ker}(\vphi\circ q)=\mathrm{Ker}(q)$. Normalizing $\psi$, we obtain the exact sequence in $\Amod_1$:
\begin{equation*}
\begin{tikzcd}
A_+\pten F \arrow[r, "\psi/\norm{\psi}_{cb}"] &A_+\pten F \arrow[r, two heads, "q"] &P.
\end{tikzcd}
\end{equation*}
Hence, $P$ is finitely presented.
\end{example}

We now show that every finitely presented module is derived through an exact sequence of matricially free modules, as promised in Remark \ref{r:1} (4). We require the following basic fact several times in the proof.

\begin{lem}\label{l:refp} Let $F$ be a finite-dimensional operator space. Then there exists $n\in\N$, $\lm\geq1$ and a complete $\lm$-quotient map $\vphi:T_n\twoheadrightarrow_\lm F$.
\end{lem}

\begin{proof} Let $n=\dim(F)$. Pick a bijective contraction $\ell^1_{n}\rightarrow F$, which is automatically a complete contraction when $\ell^1_{n}$ is equipped with its max operator space structure (see, e.g., \cite[3.3.10]{ER}). Composing this complete contraction with the canonical complete quotient map $T_{n}\twoheadrightarrow\max\ell^1_{n}$ (project to diagonal), we obtain a surjective complete contraction $\vphi:T_{n}\rightarrow F$, which, by finite-dimensionality, is a complete $\lm$-quotient map for some $\lm$ (take $\lm=n$ for instance).
\end{proof}

\begin{prop}\label{p:refp} Let $A$ be a completely contractive Banach algebra. If $E\in\Amod_1$ is (topologically) finitely presented then there exist $\lm\geq1$, $n,m\in\N$ and a (topologically) exact sequence 
\begin{equation}\label{e:refp}
\begin{tikzcd}
A_+\pten T_m \arrow[r] & A_+\pten T_n \arrow[r, two heads, "\vphi"] &E
\end{tikzcd}
\end{equation}
in $\Amod_1$ where $\vphi$ is a complete $\lm$-quotient map.
\end{prop}

\begin{proof} Suppose $E$ is topologically finitely presented via the topologically exact sequence 
\begin{equation*}
\begin{tikzcd}
A_+\pten F_1 \arrow[r, "\psi"] & A_+\pten F_2 \arrow[r, two heads, "\vphi"] &E,
\end{tikzcd}
\end{equation*}
where $\vphi$ is a complete quotient morphism and $F_1$ and $F_2$ are finite-dimensional operator spaces.  By Lemma \ref{l:refp} there exists a complete $\nu$-quotient map $T_{m_1}\rightarrow F_1$ for some $m_1\in\N$, $\nu\geq1$. Amplifying this map with $\id_{A_+}$, and composing with $\psi$ yields a morphism $\psi_1: A_+\pten T_{m_1}\rightarrow A_+\pten F_2$ whose image is dense in $\mathrm{Ker}(\vphi)$. 

By Lemma \ref{l:refp} once again, there exists a complete $\lm$-quotient map $\vphi_1:T_{n}\twoheadrightarrow F_2$ for some $\lm\geq1$ and $n\in\N$. By definition of the projective norm, its amplification $\id_{A_+}\ten\vphi_1:A_+\pten T_{n}\twoheadrightarrow A_+\pten F_2$ is also a complete $\lm$-quotient map (see \cite[Proposition 7.17]{ER} for the case $\lm=1$; its proof generalizes verbatim). By 1-projectivity of $A_+\pten T_{m_1}$ in $\Amod$ (see \ref{ss:catnotions}), there exists a completely bounded $A$-module map $\widetilde{\psi_1}:A_+\pten T_{m_1}\rightarrow A_+\pten T_{n}$ making the following diagram commute
\begin{equation*}
\begin{tikzcd}
                     &A_+\pten T_{n} \arrow[d, two heads, "\id\ten\vphi_1"]\\
A_+\pten T_{m_1} \arrow[ru, dotted, "\widetilde{\psi_1}"] \arrow[r, "\psi_1"] &A_+\pten F_2
\end{tikzcd}
\end{equation*}
(where we've implicitly used the $cb$-isomorphism $A_+\pten F_2\cong A_+\pten T_n/\mathrm{Ker}(\id\ten\vphi_1)$ to apply projectivity, hence the lack of norm control on $\widetilde{\psi_1}$).
Then 
$$\mathrm{Ker}(\vphi)=\overline{\psi_1(A_+\pten T_{m_1})}=\overline{(\id\ten\vphi_1)(\widetilde{\psi_1}(A_+\pten T_{m_1}))}.$$
Also, the cb-analogue of the kernel description for tensor products of complete quotient maps \cite[Proposition 7.1.7]{ER} implies $\mathrm{Ker}(\id\ten \vphi_1)=A_+\pten\mathrm{Ker}(\vphi_1)$. Hence
$$\mathrm{Ker}(\vphi\circ(\id\ten\vphi_1))=\overline{\widetilde{\psi_1}(A_+\pten T_{m_1})+A_+\pten\mathrm{Ker}(\vphi_1)}.$$
By finite-dimensionality of $\mathrm{Ker}(\vphi_1)$, there exists a complete $\mu$-quotient map $\psi_2:T_{m_2}\twoheadrightarrow \mathrm{Ker}(\vphi_1)$ (Lemma \ref{l:refp}) for some $\mu\geq1$ and $m_2\in\N$. Then $A_+\pten\mathrm{Ker}(\vphi_1)=(\id_{A_+}\ten\psi_2)(A_+\pten T_{m_2})$, and composing 
$$\widetilde{\psi_1}\oplus(\id_{A_+}\ten\psi_2):A_+\pten T_{m_1}\oplus_1 A_+\pten T_{m_2}=A_+\pten(T_{m_1}\oplus_1 T_{m_2})\rightarrow A_+\pten T_{n}$$
with the amplified conditional expectation $\id_{A_+}\pten T_{m_1+m_2}\twoheadrightarrow \id_{A_+}\pten (T_{m_1}\oplus_1 T_{m_2})$, we obtain a completely bounded morphism $\widetilde{\psi}:A_+\pten T_m\rightarrow A_+\pten T_{n}$ whose image is dense in $\mathrm{Ker}(\vphi\circ(\id\ten\vphi_1))$, where $m=m_1+m_2$. Upon normalization we may assume $\norm{\widetilde{\psi}}_{cb}\leq 1$, so that
\begin{equation}\label{e:fpdefn}
\begin{tikzcd}
A_+\pten T_m \arrow[r, "\widetilde{\psi}"] &A_+\pten T_{n} \arrow[r, two heads, "\vphi\circ(\id\ten\vphi_1)"] &E,
\end{tikzcd}
\end{equation}
is topologically exact in $\Amod_1$ with $\vphi\circ(\id\ten\vphi_1)$ a complete $\lm$-quotient map.

The claim for finitely presented modules follows similarly (with the closures removed).
\end{proof}

\begin{remark} It is not clear when the converse of Proposition \ref{p:refp} holds (beyond finite-dimensionality of $A$).
\end{remark}


We now show that the category $\Amod$ is locally finitely presented in the sense that any module $X\in\Amod$ is a direct limit of topologically finitely presented ones, adapting the argument from \cite[Lemma 5.39]{R}.

Let $X,Y\in\Amod$ with $Y$ a closed submodule of $X$. Let $\mc{S}$ be a set of closed submodules of $X$ partially ordered by inclusion and $\mc{T}$ be a set of closed submodules of $Y$ partially ordered by inclusion. Following \cite{R} we say that the quadruple $(X,Y,\mc{S},\mc{T})$ forms an \textit{$(X,Y)$-system} if
\begin{enumerate}
\item $\mc{S}$ and $\mc{T}$ are directed sets;
\item $X=\overline{\cup_{S\in\mc{S}} S}$ and $Y=\overline{\cup_{T\in\mc{T}} T}$;
\item for each $T\in\mc{T}$, there exists $S\in\mc{S}$ such that $S\supseteq T$. 
\end{enumerate}
Any $(X,Y)$-system gives rise to a canonical directed set
$$I:=I(X,Y,\mc{S},\mc{T})=\{(S,T)\in\mc{S}\times\mc{T}\mid S\supseteq T\},$$
where $(S_1,T_1)\leq (S_2,T_2)$ if $S_1\subseteq S_2$ and $T_1\subseteq T_2$. The family $(S/T,\vphi_{S,T,S',T'})_{(S,T),(S',T')\in I}$ forms an inductive system where the connecting maps $\vphi_{S,T,S',T'}:S/T\rightarrow S'/T'$ are the composition of the inclusion $S/T\rightarrow S'/T$ and the quotient $S'/T\rightarrow S'/T'$, whenever $(S,T)\leq (S',T')\in I$.

\begin{prop}\label{p:lfp} Let $A$ be a completely contractive Banach algebra and let $(X,Y,\mc{S},\mc{T})$ be an $(X,Y)$-system in $\Amod$. Then $\varinjlim_{(S,T)\in I}S/T\cong X/Y$ completely isometrically.
\end{prop}

\begin{proof} We show that $X/Y$ satisfies the universal property (\ref{d:co-limit}) in $\Amod$ with respect to the canonical morphisms $\vphi_{S,T}:S/T\rightarrow X/Y$. To this end, let $Z\in\Amod$ and let $(f_{S,T})_{(S,T)\in I}$ be a family of morphisms $S/T\rightarrow Z$ such that $\norm{f_{S,T}}_{cb}\leq\lm$ for all $(S,T)\in I$ and the following diagram commutes for all $(S,T)\leq(S',T')\in I$:

\begin{equation*}
\begin{tikzcd}
S/T \arrow[dd, "\vphi_{S,T,S',T'}"]\arrow[rd, "f_{S,T}"]\\
&Z\\
S'/T',\arrow[ru, "f_{S',T'}"]
\end{tikzcd}
\end{equation*} 

Given $x\in\cup_{S\in\mc{S}} S$, we have $x\in S$ for some $S\in\mc{S}$. If $T\in\mc{T}$, there exists $S'\in\mc{S}$ such that $S'\supseteq T$. Since $\mc{S}$ is directed, pick $S''\in\mc{S}$ with $S''\supseteq S,S'$. Then $(S'',T)\in I$ with $x\in S''$. Define $\vphi(x+Y)=f_{S'',T}(x+T)$. To see that $\vphi$ is well-defined, first note that the definition is independent of the chosen $f_{S'',T}$ for which $(S'',T)\in I$ and $x\in S''$ by commutativity of the above diagrams. Moreover, if $S,S'\in\mc{S}$, $x\in S$ and $x'\in S'$ with $x-x'\in Y$, for every $\ep>0$, there exists $T\in\mc{T}$ and $b\in T$ with $\norm{(x-x')-b}_Y<\ep$. There exists $S''\in\mc{S}$ with $S''\supseteq S,S',T$, so that $(S'',T)\in I$ and 
$$\norm{\vphi(x+Y)-\vphi(x'+Y)}=\norm{f_{S'',T}((x-x')+T)}\leq\lm\norm{(x-x')+T}<\ep.$$
By a similar argument as above it follows that $\norm{\vphi(x+Y)}\leq\lm\norm{x+T}$ for all $T\in\mc{T}$ such that $T\subseteq S$ for some $S\in\mc{S}$ with $x\in S$. Since $\mc{S}$ is directed and $Y=\overline{\cup_{T\in\mc{T}}T}$, it follows that $\norm{\vphi(x+Y)}\leq\lm\norm{x+Y}$. Thus, $\vphi$ extends to a well-defined bounded module map $X/Y\rightarrow Z$. It follows that $\vphi$ is actually completely bounded with $\norm{\vphi}_{cb}\leq\lm$ as each $f_{S,T}$ is so, and that $\vphi\circ\vphi_{S,T}=f_{S,T}$ for all $(S,T)\in I$. Uniqueness follows from density of $\cup_{(S,T)\in I}\vphi_{S,T}(S/T)$ in $X/Y$.
\end{proof}

\begin{thm}\label{t:lfp} Let $A$ be a completely contractive Banach algebra. Every $X\in\Amod$ is a direct limit of topologically finitely presented modules.\end{thm}

\begin{proof} Let $m^+:A_+\pten X\twoheadrightarrow X$ be the extended multiplication map. It follows that 
$$K:=\mathrm{Ker}(m^+)=\la a\cdot b\ten x-a\ten b\cdot x\mid a,b\in A_+, \ x\in X\ra,$$
and $X= A_+\pten X/K$. For finite-dimensional subspaces $\C1\subseteq E\subseteq A_+$ and $F\subseteq X$, define $X_{E,F}=\la E\cdot F\ra$ and $K_{E,F}\subseteq  A_+\pten X_{E,F}$ by
$$K_{E,F}=\la a\cdot b\ten x-a\ten b\cdot x\mid a\in A_+, \ b\in E, \ x\in F\ra.$$
Since $X_{E,F}$ is finite-dimensional, we may view $ A_+\pten X_{E,F}$ as a closed submodule of $ A_+\pten X$. The corresponding sets $\mc{S}:=\{ A_+\pten X_{E,F}\}$ and $\mc{T}:=\{K_{E,F}\}$ are directed under the canonical partial order $(E,F)\leq (E',F')$ iff $E\subseteq E'$ and $F\subseteq F'$. Clearly $ A_+\pten X=\overline{\cup_{E,F} A_+\pten X_{E,F}}$ and $K=\overline{\cup_{E,F}K_{E,F}}$. By construction, each $K_{E,F}\subseteq A_+\pten X_{E,F}$ so that $( A_+\pten X,K,\mc{S},\mc{T})$ forms an $(X,Y)$-system. Hence, by Proposition \ref{p:lfp}
$$X= A_+\pten X/K\cong\varinjlim_{E,F} A_+\pten X_{E,F}/K_{E,F}.$$

Next, observe that the map 
$$m_{ A_+}\ten\id_X-\id_{ A_+}\ten m^+_X: A_+\pten A_+\pten X\rightarrow A_+\pten X$$
restricts to a completely bounded morphism $ A_+\pten E\pten F\rightarrow K_{E,F}$ with dense range. The normalized map has the same image, so we obtain a morphism $ A_+\pten E\pten F\rightarrow K_{E,F}$ in $\Amod_1$ with dense range. The resulting sequence in $\Amod_1$
\begin{equation*}
\begin{tikzcd}
A_+\pten (E\pten F)\arrow[r] & A_+\pten X_{E,F} \arrow[r, two heads] &A_+\pten X_{E,F}/K_{E,F},
\end{tikzcd}
\end{equation*}
is topologically exact. Hence, $A_+\pten X_{E,F}/K_{E,F}$ is topologically finitely presented.

\end{proof}

It is not clear whether one can remove the word ``topologically'' from the statement of Theorem \ref{t:lfp}, in general. The difficulty stems from the fact that morphisms from finitely presented modules need not be strict when $A$ is infinite dimensional.

\begin{remark} There is a notion of finite presentation in any co-complete category, defined through the preservation of direct limits through the functor $\mathrm{Hom}(E,(\cdot))$ (see, e.g., \cite[Definition 1.1]{AR}). There is an obvious operator module analogue of this notion. One can show that any topologically finitely presented module satisfies this property, and that the notions are equivalent when $A$ is finite dimensional. We omit the (somewhat lengthy) details.
\end{remark}
 
\section{The Local Lifting Property}\label{s:LLP}

An operator space $X$ has the $\lm$-local lifting property ($\lm$-LLP) if given any operator spaces $Z\subseteq Y$ and a complete contraction $\vphi:X\rightarrow Y/Z$, for every finite-dimensional subspace $E\subseteq X$ and $\ep>0$, there exists a lifting $\widetilde{\vphi}:E\rightarrow Y$ with $\norm{\widetilde{\vphi}}_{cb}<\lm+\ep$ making the following diagram commute:
\begin{equation*}
\begin{tikzcd}
& &Y\arrow[d, two heads]\\
E\arrow[rru, dotted, "\tilde{\vphi}"]\arrow[r, hook] &X \arrow[r, "\vphi"] &Y/Z.
\end{tikzcd}
\end{equation*}
This property was studied by Kye and Ruan in \cite{KR}, where they showed, among other things, that $X$ has the $\lm$-LLP if and only if $X^*$ is $\lm$-injective \cite[Theorem 5.5]{KR}. It is also known that $X$ has the $1$-LLP if and only if for every 1-exact sequence
$$0\rightarrow Y\hookrightarrow Z\twoheadrightarrow Z/Y\rightarrow 0,$$
in $\mathbf{Op}$ the sequence
$$0\rightarrow X\pten Y\hookrightarrow X\pten Z\twoheadrightarrow X\pten Z/Y\rightarrow 0$$
is 1-exact \cite[Theorem 3.4]{Dong}. Thus, the $1$-LLP is equivalent to 1-exactness of the functor $X\pten(\cdot):\mathbf{Op}\rightarrow\mathbf{Op}$, in other words, the \textit{1-flatness} of $X$ in $\Op$. Given these results and the well-known duality between flatness and injectivity in $\Amod$, it is natural to investigate an operator module analogue of the local lifting property and its relation to flatness. 

In the purely algebraic setting, it is well-known that a right module $M$ over a unital ring $R$ is flat in $\mathbf{mod}\hskip2pt R$ if and only if for any epimorphism $\vphi:N\twoheadrightarrow M$ and any finitely presented module $E$, any homomorphism $\psi:E\rightarrow M$ can be lifted to some $\tilde{\psi}:E\rightarrow N$ (see, e.g., \cite[Corollary 4.33]{Lam}). Flatness in $\mathbf{mod}\hskip2pt R$ is therefore equivalent to a ``local'' lifting type property, where finite presentation is playing the role of locality in the category $\mathbf{mod}\hskip2pt R$. This result motivated our notion of finite presentation as well as:

\begin{defn}\label{d:LLP} Let $A$ be a completely contractive Banach algebra, and $\lm\geq 1$. A module $X\in\Amod$ has the \textit{$\lm$-local lifting property ($\lm$-LLP)} if given any $Y\in\Amod$ and a complete quotient map $q\in\Hom(Y,X)$, for every topologically finitely presented $E\in\Amod$, every morphism $\vphi\in\Hom(E,X)$ and every $\ep>0$, there exists a morphism $\tilde{\vphi}:E\rightarrow Y$ such that $\norm{\tilde{\vphi}}_{cb}<\lm\norm{\vphi}_{cb}+\ep$ and $q\circ\tilde{\vphi}=\vphi$, that is, the following diagram commutes:
\begin{equation*}
\begin{tikzcd}
&Y\arrow[d, two heads, "q"]\\
E\arrow[ru, dotted, "\tilde{\vphi}"]\arrow[r, "\vphi"] &X
\end{tikzcd}
\end{equation*}
\end{defn}

An operator space $X$ has the $\lm$-LLP in $\Op$ in the sense of Definition \ref{d:LLP} if and only if it has the operator space $\lm$-LLP by \cite[Theorem 3.2]{KR}. Hence, Definition \ref{d:LLP} is a bona fide generalization of the operator space local lifting property.

Clearly, any $\lm$-projective module $X\in\Amod$ satisfies the $\lm$-LLP. A necessary condition for the $\lm$-LLP is $\lm$-flatness, as we now show. Due to the additional quotient structure from the definition of the module projective tensor product $\pten_A$ (when $A\neq\C$), the operator space argument in \cite[Proposition 5.1]{KR} does not readily generalize. Instead, we use direct limit arguments and appeal to Theorem \ref{t:lfp}.

\begin{prop}\label{p:LLP} Let $A$ be a completely contractive Banach algebra, and $\lm\geq 1$. If $X\in\Amod$ has the $\lm$-LLP then it is $\lm$-flat.
\end{prop}

\begin{proof} Let $Y\subseteq Z$ in $\modA$, and write $i$ for the inclusion map. We show that $(i\ten\id_X):Y\pten_A X\rightarrow Z\pten_A X$ is a complete $\lm$-embedding, which will entail the $\lm$-flatness of $X$.

First, pick a set $I$ and an operator space complete quotient map $q_0:T_I\twoheadrightarrow X$ (\cite[Corollary 3.2]{Blech}). Then 
$$q:=m_{X}^+\circ(\id_{A_+}\ten q_0):A_+\pten T_I\twoheadrightarrow X$$
is a complete quotient morphism. Throughout the proof we denote $A_+\pten T_I$ by $F$, to emphasize its 1-flatness (see section \ref{ss:catnotions}).

By Theorem \ref{t:lfp}, $X\cong\varinjlim_i E_i$ completely isometrically for an inductive system $(E_i,\vphi_{j,i})_{i,j\in I}$ in $\Amod_1$ where each $E_i$ is topologically finitely presented. Since $X$ has the $\lm$-LLP, for each $i\in I$ and every $\ep>0$ there exists a morphism $\widetilde{\vphi_i}:E_i\rightarrow F$ such that $\norm{\widetilde{\vphi_i}}_{cb}<\lm+\ep$ and the following diagram commutes:
\begin{equation*}
\begin{tikzcd}
&F\arrow[d, two heads, "q"]\\
 E_i\arrow[ru, "\widetilde{\vphi_i}"]\arrow[r, "\vphi_i"] &X,
\end{tikzcd}
\end{equation*}
where $\vphi_i:E_i\rightarrow X$ is the canonical morphism. Then for $u\in M_n(Y\pten_A E_i)$, 
\begin{align*}\norm{(\id\ten\vphi_i)_n(u)}_{M_n(Y\pten_A X)}&=\norm{(\id\ten q\circ\widetilde{\vphi_i})_n(u)}_{M_n(Y\pten_A X)}\\
&\leq\norm{(\id\ten \widetilde{\vphi_i})_n(u)}_{M_n(Y\pten_A F)}\\
&=\norm{(\id\ten \widetilde{\vphi_i})_n((i\ten \id)(u))}_{M_n(Z\pten_A F)}\\
&\leq(\lm+\ep)\norm{(i\ten \id)(u)}_{M_n(Z\pten_A E_i)},
\end{align*}
where the second equality follows from 1-flatness of $F$. Since $\ep>0$ was arbitrary, 
$$\norm{(\id\ten\vphi_i)_n(u)}_{M_n(Y\pten_A X)}\leq\lm\norm{(i\ten\id)(u)}_{M_n(Z\pten_A E_i)}.$$
Write $Y\pten_A^Z X$ for the closure of $(i\ten\id_X)(Y\pten_A X)$ in $Z\pten_A X$, and similarly for each $Y\pten_A^Z E_i$. By above, the map $(\id\ten\vphi_i)$ extends to a morphism 
$$\widetilde{(\id\ten\vphi_i)}:Y\pten_A^Z E_i\rightarrow Y\pten_A X$$
in $\Amod$ with $\norm{\widetilde{(\id\ten\vphi_i)}}_{cb}\leq\lm$. Moreover, for each $i,j\in I$ the following diagram commutes
\begin{equation*}
\begin{tikzcd}
Y\pten_A^Z E_i\arrow[dd, "(\id_Z\ten \vphi_{j,i})"]\arrow[rd, "\widetilde{(\id\ten \vphi_i)}"]\\
 &Y\pten_A X\\
Y\pten_A^Z E_j\arrow[ru, "\widetilde{(\id\ten \vphi_j)}"] 
\end{tikzcd}
\end{equation*}
(This is easily checked on the image $(i\ten\id)(Y\pten_A E_i)$, so the result follows by density.) By the universal property of direct limits, there is a unique morphism $\vphi:\varinjlim_i Y\pten_A^Z E_i\rightarrow Y\pten_A X$ with $\norm{\vphi}_{cb}\leq\lm$ satisfying
$$\vphi((i\ten\id)(\id\ten\vphi_i)(u_i))=(\id\ten\vphi_i)(u_i), \ \ \ u_i\in Y\pten_A E_i, \ i\in I.$$
Since each $Y\pten_A^Z E_i\hookrightarrow Z\pten_A E_i$ completely isometrically, it follows from the representation in Proposition \ref{p:limitrep} that the induced mapping between the direct limits 
$$\varinjlim_i Y\pten_A^Z E_i\hookrightarrow\varinjlim_i Z\pten_A E_i$$
is completely isometric. Furthermore, by Proposition \ref{p:iso} $\varinjlim_i Z\pten_A E_i\cong Z\pten_A X$, canonically. It follows that $\varinjlim_i Y\pten_A^Z E_i=Y\pten_A^Z X$, canonically. Under this identification, $\vphi:Y\pten_A^Z X\rightarrow Y\pten_A X$ is a left inverse to $(i\ten\id)$ with $\norm{\vphi}_{cb}\leq\lm$. Hence, $(i\ten\id)$ is a complete $\lm$-embedding.
\end{proof}

We now establish a partial converse to Proposition \ref{p:LLP} under the assumption that $A_+\pten T_I$ is locally reflexive for any set $I$. By \cite{EJR}, this condition is satisfied whenever $A$ is the predual of a von Neumann algebra, e.g., $A=\LOQ$ for any locally compact quantum group $\G$. In what follows we use the notation $\simeq_b$ to mean isomorphism in the category $\mathbf{Ban}$ of Banach spaces and bounded linear maps.

For $X,Y\in\Amod$ there is a canonical complete contraction
$$\Phi:X^*\pten_A Y\ni f\ten_A y\mapsto (\vphi\mapsto\la f,\vphi(y)\ra) \in\Hom(Y,X)^*.$$
When $A=\C$ and $Y=E$ is finite dimensional, then $\Hom(E,X)=E^*\iten X$ (injective operator space tensor product) and the map $\Phi$ is the canonical isomorphism
$$X^*\pten E=X^*\pten E^{**}\simeq_b(E^*\iten X)^*,$$
which appears in the study of local reflexivity (see, e.g., \cite[\S14.3]{ER}). Such an isomorphism holds for general $A$ whenever $E$ is topologically finitely presented, as we now show. This may be seen as an operator module analogue of (a special case of) the purely algebraic result \cite[Lemma 9.71]{R}.

\begin{prop}\label{p:fp} Let $A$ be a completely contractive Banach algebra, and let $E\in\Amod$ be topologically finitely presented. Then $X^*\pten_A E\simeq_b\Hom(E,X)^*$ for all $X$ in $\Amod$.
\end{prop}

\begin{proof} First suppose that $E= A_+\pten F$ for some finite-dimensional operator space $F$. By associativity of module tensor products \cite[Theorem 3.4.10]{BLM} and equation (\ref{e:A_+})
$$X^*\pten_A E=X^*\pten_A ( A_+ \pten F)=(X^*\pten_A  A_+ )\pten F\cong X^*\pten F\simeq_b\mc{CB}(F,X)^*\cong\Hom(E,X)^*.$$
Now, suppose that we have a topologically exact sequence in $\Amod_1$
\begin{equation*}
\begin{tikzcd}
A_+ \pten F_1 \arrow[r] &A_+ \pten F_2 \arrow[r, two heads, "\vphi"] &E,
\end{tikzcd}
\end{equation*}
with $\vphi$ a complete quotient morphism. Then by Lemma \ref{l:rightexact}
\begin{equation*}
\begin{tikzcd}
X^*\pten_A( A_+ \pten F_1)\arrow[r] & X^*\pten_A( A_+ \pten F_2)\arrow[r, two heads] & X^*\pten_A E
\end{tikzcd}
\end{equation*}
is topologically exact in $\Op$, and 
\begin{equation*}
\begin{tikzcd}
\Hom(E,X)\arrow[r, hook] &\Hom( A_+ \pten F_2,X)\arrow[r] &\Hom( A_+ \pten F_1,X)
\end{tikzcd}
\end{equation*}
is exact in $\Op$. Since $\C$ is injective in $\Op$, the sequence
\begin{equation*}
\begin{tikzcd}
\Hom( A_+ \pten F_1,X)^* \arrow[r] &\Hom( A_+ \pten F_2,X)^* \arrow[r, two heads] &\Hom(E,X)^*
\end{tikzcd}
\end{equation*}
is also exact. Consider the commutative diagram
\begin{equation}\label{e:diag1}
\begin{tikzcd}
& X^*\pten_A( A_+ \pten F_1)\arrow[r]\arrow[d, equal] & X^*\pten_A( A_+ \pten F_2)\arrow[r, two heads]\arrow[d, equal] & X^*\pten_A E\arrow[d, "\Phi"]\\
& \Hom( A_+ \pten F_1,X)^*\arrow[r] & \Hom( A_+ \pten F_2,X)^*\arrow[r,two heads] &\Hom(E,X)^*.
\end{tikzcd}
\end{equation}
Since the top row is topologically exact (specifically, its first arrow maps into the kernel of its second), the bottom row is exact, the first two vertical arrows are (bounded) isomorphisms, and the second arrows in each row are surjective, the usual diagram chase implies that $\Phi$ is bijective, and thus a bounded isomorphism by the inverse mapping theorem.
\end{proof}

Proposition \ref{p:fp} suggests a notion of $\lm$-local reflexivity for modules $X\in\Amod$: every completely contractive morphism $\vphi:E\rightarrow X^{**}$ from a topologically finitely presented module $E$ can be approximated in the point weak* topology by a net $\vphi_i:E\rightarrow X$ satisfying $\norm{\vphi_i}_{cb}\leq \lm$. Equivalently, the isomorphism
$$\Phi:X^*\pten_A E\simeq_b\Hom(E,X)^*$$
satisfies $\norm{\Phi^{-1}}\leq\lm$ for every topologically finitely presented $E\in\Amod$. However, this notion coincides with $\lm$-local reflexivity in $\Op$.

\begin{prop}\label{p:locrefl} Let $A$ be a completely contractive Banach algebra, and $\lm\geq 1$. Then $X\in\Amod$ is $\lm$-locally reflexive in $\Amod$ (in the sense above) if and only if it is $\lm$-locally reflexive in $\Op$.
\end{prop}

\begin{proof} If $X$ is $\lm$-locally reflexive in $\Amod$, then for every finite-dimensional $E\in\Op$, $A_+\pten E$ is finitely presented, so we have
$$X^*\pten E\cong X^*\pten_A ( A_+ \pten E)\simeq_b\Hom( A_+ \pten E,X)^*\cong\mc{CB}(E,X)^*=(E^*\iten X)^*,$$
with the norm of the inverse bounded by $\lm$. Thus, $X$ is $\lm$-locally reflexive in $\Op$.

Conversely, suppose $X$ is $\lm$-locally reflexive in $\Op$. Then for any finite-dimensional operator space $F$,  
$$X^*\pten_A ( A_+ \pten F)\cong X^*\pten F\simeq_b\mc{CB}(F,X)^*\cong\Hom( A_+ \pten F,X)^*,$$
with the norm of the inverse bounded by $\lm$. If $E\in\Amod$ is topologically finitely presented, then there is a topologically exact sequence
\begin{equation*}
\begin{tikzcd}
A_+\pten F_1 \arrow[r] & A_+\pten F_2 \arrow[r, two heads, "\vphi"] &E
\end{tikzcd}
\end{equation*}
in $\Amod_1$ with $\vphi$ a complete quotient map. Building the corresponding diagram as in (\ref{e:diag1}), the first row is topologically exact, the second row is exact, the first two vertical arrows have inverses bounded by $\lm$, and the second horizontal arrows in each row are complete quotient maps. It follows that $$X^*\pten_A E\simeq_b\Hom(E,X)^*$$ with inverse bounded by $\lm$.
\end{proof}

The previous two results allow for an operator module extension of (the first part of) \cite[Proposition 5.4]{KR}. Contrary to Proposition \ref{p:LLP}, the operator space argument from \cite{KR} extends more or less verbatim.

\begin{prop}\label{p:KR} Let $Y,Z\in\Amod$, $Z\subseteq Y$ and $Y$ be locally reflexive in $\Op$. If $Z^{\perp}$ is $\lm$-completely complemented as a submodule of $Y^*$, then
\begin{enumerate}
\item $Y/Z$ is $\lm$-locally reflexive, and
\item the canonical map 
$$T:\Hom(E,Y)/\Hom(E,Z)\rightarrow\Hom(E,Y/Z)$$
satisfies $\norm{T^{-1}}\leq\lm$ for every topologically finitely presented $E\in\Amod$.
\end{enumerate}
\end{prop}

\begin{proof} Let $E\in\Amod$ be topologically finitely presented. Let $q:Y\twoheadrightarrow Y/Z$ and 
$$Q:\Hom(E,Y)\rightarrow \Hom(E,Y)/\Hom(E,Z)$$
be the canonical quotient maps, so that $\Hom(E(q))=T\circ Q$, where $\Hom(E(\cdot))$ is the standard functor. Write $i$ for the inclusion $Z^{\perp}\subseteq Y^*$, and let $Z^{\perp}\pten_{A}^{Y^*}E$ denote the closed subspace generated by $(i\ten\id)(Z^\perp\pten_A E)$. We will show that the following diagram
\begin{equation}\label{e:diag2}
\begin{tikzcd}
\Hom(E,Y/Z)^*  \arrow[rr, "\mathrm{Hom}(E(q))^*"] \arrow[rd, "T^*"] &  &\Hom(E,Y)^*\\
 & \bigg(\Hom(E,Y)/\Hom(E,Z)\bigg)^* \arrow[ru, "Q^*"] &\\
Z^{\perp}\pten_A E \arrow[uu, "\Phi_1"] \arrow[r, "(i\ten\id)"] &Z^{\perp}\pten_{A}^{Y^*}E \arrow[r, hook]\arrow[u, "\Phi_2"] &Y^*\pten_A E \arrow[uu, "\Phi"]
\end{tikzcd}
\end{equation}
commutes, where $\Phi_2$ is the restriction of $\Phi$ to $Z^{\perp}\pten_{A}^{Y^*}E$. For any $f\in Z^{\perp}$, $x\in E$, and $\vphi\in\Hom(E,Y)$, we have
\begin{align*}\la\Hom(E(q))^*\circ\Phi_1(f\ten_A x),\vphi\ra&=\la\Phi_1(f\ten_A x),q\circ\vphi\ra=\la f,q(\vphi(x))\ra\\
&=\la q^*(f),\vphi(x)\ra=\la\Phi(q^*(f)\ten_A x),\vphi\ra\\
&=\la\Phi((i\ten\id)(f\ten_A x)),\vphi\ra.
\end{align*}
Hence, the outer rectangle in (\ref{e:diag2}) commutes, that is  
$\Hom(E(q))^*\circ\Phi_1=\Phi\circ (i\ten\id)$. 
Since $Q^*$ is just the inclusion $\Hom(E,Z)^{\perp}\subseteq\Hom(E,Y)^*$, and $\Hom(E(q))^*=Q^*\circ T^*$, it follows that $T^*\circ\Phi_1=\Phi_2\circ (i\ten\id)$. Hence, the entire diagram (\ref{e:diag2}) commutes.

Now, since $Z^{\perp}$ is a $\lm$-completely complemented submodule of $Y^*$, $\norm{(i\ten\id)^{-1}}\leq\lm$, and since $Y$ is locally reflexive, Proposition \ref{p:locrefl} implies that $\norm{\Phi^{-1}}\leq 1$. Thus, $\Phi$ is an isometric isomorphism, implying its restriction $\Phi_2$ is an isometric isomorphism onto its range, which is all of $(\Hom(E,Y)/\Hom(E,Z))^*$ by commutativity of (\ref{e:diag2}) and the fact that $Z^{\perp}$ is completely complemented in $Y^*$. Hence, $\norm{\Phi_2^{-1}}\leq 1$, and 
$$\norm{\Phi_1^{-1}}=\norm{(i\ten\id)^{-1}\circ\Phi_2^{-1}\circ T^*}\leq\norm{(i\ten\id)^{-1}}\leq\lm.$$
Since $E\in\Amod$ was arbitrary, it follows that $Y/Z$ is $\lm$-locally reflexive (in either category, by Proposition \ref{p:locrefl}). Finally,
$$\norm{T^{-1}}=\norm{(T^*)^{-1}}=\norm{\Phi_1\circ (i\ten\id)^{-1}\circ\Phi_2^{-1}}\leq\norm{(i\ten\id)^{-1}}\leq\lm.$$
\end{proof}


\begin{thm}\label{t:LLP} Let $A$ be a completely contractive Banach algebra such that $A_+\pten T_I$ is locally reflexive for every set $I$, and let $\lm\geq 1$. Then $X\in\Amod$ is $\lm$-flat if and only if $X$ has the $\lm$-LLP.
\end{thm}

\begin{proof} The reverse direction follows immediately from Proposition \ref{p:LLP}. 

Suppose that $X$ is $\lm$-flat in $\Amod$, equivalently, $X^*$ is $\lm$-injective in $\modA$. As in the proof of Proposition \ref{p:LLP}, $X=  A_+ \pten T_I/Z$ for some set $I$ and closed submodule $Z\subseteq A_+ \pten T_I$. By assumption, $A_+ \pten T_I$ is locally reflexive and $Z^{\perp}=X^*$ is a $\lm$-completely complemented submodule of $( A_+ \pten T_I)^*$. Proposition \ref{p:KR} implies that the canonical map 
$$T:\Hom(E, A_+ \pten T_I)/\Hom(E,Z)\rightarrow\Hom(E,X)$$
satisfies $\norm{T^{-1}}\leq \lm$ for any topologically finitely presented $E\in\Amod$. This implies that $X$ has the $\lm$-LLP with respect to any quotient of the form $A_+\pten T_I\twoheadrightarrow X$. Since any module in $\Amod$ is a complete quotient of some $A_+\pten T_I$, it follows that $X$ has the $\lm$-LLP.
\end{proof}

\begin{remark} Local type characterizations of (relatively) flat Banach modules were studied by Aristov in \cite{A}. Among other things, he showed that flatness of a Banach $A$-module $X$ is equivalent to a different type of local lifting property, namely with respect to finite-dimensional subspaces of $X$ of the form $E=\mathrm{span}\{a_i\cdot x_j\mid i=1,...,m, \ j=1,...,n\}$ with $a_1,...,a_m\in A_+$ and $x_1,...,x_n\in X$ (see \cite[Theorem 2.1]{A}). However, the morphisms involved in this property are only assumed to be approximate $A$-module maps, and therefore live outside $\Amod$. His work inspired us to seek a local characterization of flatness \textit{within} $\Amod$. Theorem \ref{t:LLP} achieves this and shows that our concept of topological finite presentation provides a suitable notion of locality in this context.
\end{remark}

\begin{cor}\label{c:LLPproj} Let $A$ be a completely contractive Banach algebra such that $A_+\pten T_I$ is locally reflexive for every set $I$, and $\lm\geq 1$. If $X\in\Amod$ is $\lm$-flat and topologically finitely presented then it is $\lm$-projective.
\end{cor}

\begin{proof} Simply apply Theorem \ref{t:LLP} together with the fact that $X$ is topologically finitely presented.
\end{proof}

As previously mentioned, the local reflexivity assumption in Theorem \ref{t:LLP} is satisfied whenever $A$ is the predual of a von Neumann algebra by \cite{EJR} (and \cite[Theorem 7.2.4]{ER}). In particular, when $A=\LOQ$ for a locally compact quantum group $\G$. Combining Theorem \ref{t:LLP} with \cite[Theorem 5.1]{C} (which equates $\LOQ$-injectivity of $\LIQ$ with amenability of $\h{\G}$), we obtain a local characterization of flatness of $\LOQ$ in the module category $\LOQmod$, which is new even for groups.

\begin{cor}\label{c:1} Let $\G$ be a locally compact quantum group. Then $\h{\G}$ is amenable if and only if $\LOQ$ has the 1-LLP in $\LOQmod$.
\end{cor}

Thus, for a locally compact group $G$, $\LO$ always has the 1-LLP in $\LO\hskip2pt\mathbf{mod}$, and $A(G)$ has the 1-LLP in $A(G)\hskip2pt\mathbf{mod}$ if and only if $G$ is amenable.

\section{Nuclearity and Semi-discreteness}\label{s:nuc}

An operator space $X$ is \textit{nuclear} if there exist diagrams of complete contractions 

\begin{equation*}
\begin{tikzcd}
                     &M_{n_\alpha} \arrow[dr, "s_\alpha"]\\
X \arrow[ru, "r_\alpha"] \arrow[rr, "\id"] & &X
\end{tikzcd}
\end{equation*}
which approximately commute in the point-norm topology \cite[\S14.6]{ER}. In other words, the identity map on $X$ approximately factorizes through finitely co-generated co-free operator spaces. From this categorical perspective, the following definition is natural.

\begin{defn}\label{d:nuc} Let $A$ be a completely contractive Banach algebra, and $X,Y\in\Amod_1$. 
\begin{itemize}
\item A morphism $\vphi\in\Hom(X,Y)$ is \textit{nuclear} (or \textit{$A$-nuclear}) if there exist diagrams of morphisms

\begin{equation*}
\begin{tikzcd}
                     &M_{n_\alpha}(A_+^*) \arrow[dr, "s_\alpha"]\\
X \arrow[ru, "r_\alpha"] \arrow[rr, "\vphi"] & &Y
\end{tikzcd}
\end{equation*}
which approximately commute in the point-norm topology. 
\item An operator module $X\in\Amod_1$ is \textit{nuclear} (or \textit{$A$-nuclear}) if $\id_X\in\Hom(X,X)$ is nuclear. 
\end{itemize}
Similar definitions apply to right modules and their morphisms.
\end{defn}

Recently, a generalization of the weak expectation property was introduced in \cite{BC2} for operator modules over completely contractive Banach algebras. An object $X\in\Amod_1$ has the \textit{weak expectation property ($A$-WEP)} if for any completely isometric morphism $\kappa: X\hookrightarrow Y$ there exists a morphism $\psi:Y\rightarrow X^{**}$ such that $\psi\circ\kappa =i_X$, where $i_X:X\hookrightarrow X^{**}$ is the canonical inclusion. Any nuclear module has $A$-WEP, a consequence of:

\begin{prop}\label{p:nucWEP} Let $A$ be a completely contractive Banach algebra, and $X\in\Amod_1$. If the inclusion $i_X:X\hookrightarrow X^{**}$ is nuclear then $X$ has the $A$-WEP. In particular, nuclearity implies the WEP in $\Amod_1$.
\end{prop}

\begin{proof} Pick diagrams of morphisms

\begin{equation*}
\begin{tikzcd}
                     &M_{n_\alpha}(A_+^*) \arrow[dr, "s_\alpha"]\\
X \arrow[ru, "r_\alpha"] \arrow[rr, "i_{X}"] & &X^{**}
\end{tikzcd}
\end{equation*}
which approximately converge in the point-norm topology. Suppose $X\subseteq\BH$ in $\Op$, and let $\kappa:X\hookrightarrow\mc{CB}(A_+,\BH)$ be the completely isometric composition of this inclusion with the canonical embedding $\Delta_+:X\hookrightarrow\mc{CB}(A_+,X)$. By injectivity of the co-free module $M_{n_\alpha}(A_+^*)$, each $r_\alpha$ extends to a morphism $\widetilde{r_\alpha}:\mc{CB}(A_+,\BH)\rightarrow M_{n_\alpha}(A_+^*)$, for which $\widetilde{r_\alpha}\circ\kappa=r_\al$. Let $\vphi\in\mc{CB}(\mc{CB}(A_+,\BH),X^{**})$ be a weak*-cluster point of $(s_\al\circ \widetilde{r_\alpha})$. Then $\vphi\circ\kappa=i_X$. Hence, the inclusion $X\subseteq X^{**}$ factors through the co-free, hence injective module $\mc{CB}(A_+,\BH)$, so $X$ has the $A$-WEP by (the left version of) \cite[Theorem 3.4(6)]{BC2}.
\end{proof}

We have the analogous notion and result for dual modules.

\begin{defn}\label{d:semid} Let $A$ be a completely contractive Banach algebra, and $X,Y\in\Amod_1$ be dual modules. 
\begin{itemize}
\item A normal (i.e., weak*-weak* continuous) morphism $\vphi\in\Hom(X,Y)$ is \textit{weakly nuclear} (or \textit{weakly $A$-nuclear}) if there exist diagrams of normal morphisms

\begin{equation*}
\begin{tikzcd}
                     &M_{n_\alpha}(A_+^*) \arrow[dr, "s_\alpha"]\\
X \arrow[ru, "r_\alpha"] \arrow[rr, "\vphi"] & &Y
\end{tikzcd}
\end{equation*}
which approximately commute in the point-weak* topology. 
\item A dual module $X\in\Amod_1$ is \textit{semi-discrete} (or \textit{$A$-semi-discrete}) if $\id_X\in\Hom(X,X)$ is weakly nuclear. 
\end{itemize}
Similar definitions apply to dual right modules and their normal morphisms.
\end{defn}

As expected, any semi-discrete module is injective.

\begin{prop}\label{p:semiinj} Let $A$ be a completely contractive Banach algebra and $X\in\Amod_1$ be a dual module. If $X$ is semi-discrete then it is injective in $\Amod_1$.
\end{prop}

\begin{proof} Pick diagrams 
\begin{equation*}
\begin{tikzcd}
                     &M_{n_\alpha}(A_+^*) \arrow[dr, "s_\alpha"]\\
X \arrow[ru, "r_\alpha"] \arrow[rr, "\id_X"] & &X
\end{tikzcd}
\end{equation*}
of normal morphisms in $\Amod_1$ which approximately commute in the point-weak* topology. Also, pick an embedding $X\hookrightarrow\mc{CB}(A_+,\BH)$ in $\Amod_1$ for some Hilbert space $H$ (as above). By injectivity of $M_{n_\alpha}(A_+^*)$, each $r_\alpha$ extends to a morphism $\widetilde{r_\alpha}:\mc{CB}(A_+,\BH)\rightarrow M_{n_\alpha}(A_+^*)$. Let $\Phi\in\Hom(\mc{CB}(A_+,\BH),X)$ be weak*-cluster point of $(s_\alpha\circ\widetilde{r_\alpha})$. Then $\Phi$ is a module projection of norm 1, witnessing the injectivity of $X$.
\end{proof}

Clearly, any module of the form $M_n(A_+^*)$ is both $A$-nuclear and $A$-semi-discrete. Another simple class of examples follows from 

\begin{prop}\label{p:nuc} Let $A$ be a completely contractive Banach algebra. If $A$ has a contractive left (respectively, right) approximate identity, then $\la A\cdot A^*\ra$ (respectively, $\la A^*\cdot A\ra$) is nuclear and $A^*$ is semi-discrete in $\modA_1$ (respectively, $\Amod_1$).
\end{prop}

\begin{proof} Viewing $A_+$ in $\Amod_1$, the right module structure on $A_+^*=A^*\oplus_\infty \C$ is given by
$$(x,\lm)\cdot a=(x\cdot a,\la x,a\ra), \ \ \ a\in A, \ x\in A^*, \ \lm\in\C.$$
It follows that the projection $p:A_+^*\ni(x,\lm)\mapsto x\in A^*$ is a morphism in $\modA_1$.

Let $(a_\alpha)$ be a contractive left approximate identity for $A$. Then we obtain 
diagrams of morphisms in $\modA_1$

\begin{equation*}
\begin{tikzcd}
                     &A_+^* \arrow[dr, "s_\alpha"]\\
\la A\cdot A^*\ra \arrow[ru, "r_\alpha"] \arrow[rr, "\id"] & &\la A\cdot A^*\ra,
\end{tikzcd}
\end{equation*}
where $r_\alpha(x)=(a_\alpha\cdot x,\la x,a_\alpha\ra)$, and $s_\alpha(x,\lm)=l(a_\alpha)\circ p(x,\lm)=a_{\alpha}\cdot x$. The composition $s_\alpha\circ r_\alpha(x)=a_\alpha^2\cdot x\rightarrow x$ for all $x\in \la A\cdot A^*\ra$. Hence, $\la A\cdot A^*\ra$ is nuclear in $\modA_1$. Similarly, $s_\alpha\circ r_\alpha(x)=a_\alpha^2\cdot x\rightarrow x$ weak* for all $x\in A^*$. 

The argument for the right case is identical.
\end{proof}

Letting $A=\LOQ$ in Proposition \ref{p:nuc} for a co-amenable locally compact quantum group $\G$, we see that both $\mathrm{LUC}(\G):=\la\LIQ\star\LOQ\ra$ and $\mathrm{RUC}(\G):=\la\LOQ\star\LIQ\ra$ are $\LOQ$-nuclear, and $\LIQ$ is $\LOQ$-semi-discrete. In particular, for any locally compact group $G$, $\mathrm{LUC}(G)$ and $\mathrm{RUC}(G)$ are $\LO$-nuclear and $\LI$ is $\LO$-semi-discrete. Dually, for an amenable locally compact group $G$, $\mathrm{UCB}(\h{G})$ is $A(G)$-nuclear and $VN(G)$ is $A(G)$-semi-discrete, where  $\mathrm{UCB}(\h{G})=\la A(G)\cdot VN(G)\ra$ is the space of uniformly continuous linear functionals on $A(G)$, introduced by Granirer \cite{Gran}. When $G$ is discrete, $\mathrm{UCB}(\h{G})=C^*_\lm(G)$, so $C^*_\lm(G)$ is $A(G)$-nuclear for any discrete amenable group $G$. 

We now establish the converse of Proposition \ref{p:nuc} for any locally compact quantum group. The next two results are helpful in this regard.

\begin{lem}\label{l:nuc} Let $X\in\modA_1$ be faithful. Any morphism $A_+^*\rightarrow X$ in $\modA_1$ is of the form $\vphi\circ p$, where $p:A_+^*\rightarrow A^*$ is the canonical projection and $\vphi\in\Hom(A^*,X)$. If, in addition, $X$ is a dual module and $A_+^*\rightarrow X$ is normal, the resulting morphism $\vphi\in\Hom(A^*,X)$ is normal.
\end{lem}

\begin{proof} Let $\psi\in\Hom(A_+^*,X)$. Since $\psi$ is $\C$-linear we can write
$$\psi(x,\lm)=\psi(x,0)+\psi(0,\lm)=\vphi(x)+\lm x_\psi,$$
where $\vphi=\psi|_{A^*\oplus 0}\in\mc{CB}(A^*,X)$ and $\psi|_{0\oplus\C}\in\mc{CB}(\C,X)$ is uniquely determined by $x_\psi\in X$. Then for every $\lm\in \C$ and $a\in A$,
$$0=\psi((0,\lm)\cdot a)=\psi(0,\lm)\cdot a=\lm x_\psi\cdot a.$$
Since $X$ is faithful, this forces $x_\psi=0$. Thus, $\psi(x,\lm)=\vphi(x)=\vphi\circ p(x,\lm)$, and $\vphi$ is necessarily a morphism. 

The normality statement is immediate.
\end{proof}

Recall that a locally compact quantum group $\G$ has the \textit{approximation property} if there exists a net $(\hat{f}_i')$ in $\LOQHP$ such that $\h{\Theta'}^{r}(\hat{f}'_i)$ converges to $\id_{\LIQHP}$ in the stable point-weak* topology, meaning $\h{\Theta'}^{r}(\hat{f}'_i)\ten\id_{M_\infty}\rightarrow\id_{\LIQHP\oten M_\infty}$ point-weak*, where $M_\infty=M_{\N}$ in the notation of section \ref{ss:op}. See \cite{ER4} for details on this topology for dual operator spaces, and \cite{KR2} for the approximation property of Kac algebras.

\begin{prop}\label{p:Wepamen} Let $\G$ be a locally compact quantum group. Consider the following conditions:
\begin{enumerate}
\item $\G$ is co-amenable;
\item $C_0(\G)$ has the WEP in $\modLOQ_1$;
\item $\h{\G}$ is amenable.
\end{enumerate}
\end{prop}
Then $(1)\Rightarrow(2)\Rightarrow(3)$. When $\G$ is co-commutative, or if $\h{\G}$ has the approximation property, the conditions are equivalent.

\begin{proof} $(1)\Rightarrow(2)$: By Proposition \ref{p:nuc}, $(1)$ implies that $\RUC=\la\LOQ\star\LIQ\ra$ is nuclear in $\modLOQ_1$. Since 
$$C_0(\G)\subseteq\RUC\subseteq M(C_0(\G))\subseteq C_0(\G)^{**}$$
\cite{Runde}, it follows that the inclusion $C_0(\G)\hookrightarrow C_0(\G)^{**}$ is nuclear, so (the right module version of) Proposition \ref{p:nucWEP} entails $(2)$.

$(2)\Rightarrow(3)$: If $C_0(\G)$ has the WEP in $\modLOQ_1$, then there exists a morphism $\vphi:\BLTQ\rightarrow C_0(\G)^{**}$ such that $\vphi|_{C_0(\G)}=i_{C_0(\G)}$. Here, the pertinent $\LOQ$-module structure on $\BLTQ$ is
$$T\star f=(f\ten\id)W^*(1\ten T)W, \ \ \ f\in\LOQ, \ T\in\BLTQ.$$
Then $\vphi$ is a weak expectation, and hence, by \cite[Lemma 2.3]{BC2}, a completely positive $M(C_0(\G))$-bimodule map. Composing with the canonical morphism $C_0(\G)^{**}\twoheadrightarrow\LIQ$, we obtain a completely positive $\LOQ$-morphism $\vphi:\BLTQ\rightarrow\LIQ$ which is the identity on $C_0(\G)$ and is also an $M(C_0(\G))$-bimodule map. The $\LOQ$-module property implies that $\vphi(\LIQHP)\subseteq \C1$: since $W\in\LIQ\oten\LIQH$, for any $\hat{x}'\in\LIQHP$ and any $f\in\LOQ$,
\begin{align*}\la f,1\ra\vphi(\hat{x}')&=\vphi((f\ten\id)(1\ten\hat{x}'))\\
&=\vphi((f\ten\id)W^*(1\ten\hat{x}')W)\\
&=\vphi(\hat{x}'\star f)\\
&=\vphi(\hat{x}')\star f.
\end{align*}
It follows that $\Gam(\vphi(\hat{x}'))=1\ten \vphi(\hat{x}')$, which, by \cite[Result 5.13]{KV1} implies $\vphi(\hat{x}')\in\C1$. In particular, $\vphi(1)=\lm1$ for some $\lm\geq0$. But then complete positivity implies
$$\lm=\norm{\lm1}=\norm{\vphi(1)}=\norm{\vphi}_{cb}=1,$$
so $\vphi$ is unital. 

Since the left fundamental unitary $\h{W}'$ of $\h{\G}'$ is the right fundamental unitary $V$ of $\G$ (see \cite[\S4]{KV2}), and $V\in M(C_0(\h{\G}')\iten C_0(\G))$ (see, e.g., \cite[Theorem 1.5]{Wo}), we have
$$\h{W}'\in M(C_0(\h{\G}')\iten C_0(\G))\subseteq\LIQHP\oten\BLTQ.$$
It follows (similar to \cite[Theorem 3.2]{NV}) that
$$(\hat{f}'\ten\id)(\id\ten\vphi)(\h{W}')=\vphi(((\hat{f}'\ten\id)(\h{W}'))=(\hat{f}\ten\id)(\h{W}'), \ \ \ \hat{f}'\in\LOQHP,$$
as $\vphi$ is unital and $(\hat{f}'\ten\id)(\h{W}')\in M(C_0(\G))$. Hence, $(\id\ten\vphi)(\h{W}')=\h{W}'$, and by unitarity, $\h{W}'$ is in the multiplicative domain of $(\id\ten\vphi)$. The bimodule property of completely positive maps over their multiplicative domains implies that
\begin{align*}\vphi((\hat{f}'\ten\id)(\h{W}'^*(1\ten T)\h{W}'))&=(\hat{f}'\ten\id)(\id\ten\vphi)((\h{W}'^*(1\ten T)\h{W}'))\\
&=(\hat{f}'\ten\id)(\h{W}'^*(1\ten \vphi(T))\h{W}'),
\end{align*}
for all $T\in\BLTQ$ and $\hat{f}'\in\LOQHP$. Thus, $\vphi|_{\LIQHP}:\LIQHP\rightarrow\C$ satisfies
$$\vphi|_{\LIQHP}(\hat{x}'\hat{\star}'\hat{f}')=\vphi|_{\LIQH}(\hat{x}')\hat{\star}'\hat{f}', \ \ \ \hat{x}'\in\LIQHP, \ \hat{f}'\in\LOQHP,$$
so is a left invariant mean, and $\h{\G}'$ is amenable. The standard argument shows that $\h{\G}$ is amenable, which we include for convenience. 

Let $R:\BLTQ\rightarrow\BLTQ$ be the *-anti-automorphism $R(T)=\h{J}T^*\h{J}$, where $\h{J}$ is the modular conjugation of the dual Haar weight $\h{h_L}$. Then $R(\LIQH)=\h{J}\LIQH\h{J}=\LIQHP$, and by \cite[Proposition 2.15]{KV2}
$$(R\ten R)(\h{W})=(\h{J}\ten\h{J})\h{W}^*(\h{J}\ten\h{J})=V^*.$$
Let $\hat{m}=\vphi|_{\LIQHP}\circ R\in\LIQH^*$. Then $\hat{m}$ is a state satisfying
\begin{align*}\hat{m}(\hat{x}\star\hat{f})&=\hat{m}((\hat{f}\ten\id)\h{W}^*(1\ten\hat{x})\h{W})\\
&=\vphi|_{\LIQHP}((\hat{f}\circ R\ten\id)(R\ten R)(\h{W}^*(1\ten\hat{x})\h{W}))\\
&=\vphi|_{\LIQHP}((\hat{f}\circ R\ten\id)(R\ten R)(\h{W})(1\ten R(\hat{x}))(R\ten R)(\h{W}^*))\\
&=\vphi|_{\LIQHP}((\hat{f}\circ R\ten\id)V^*(1\ten R(\hat{x}))V)\\
&=\vphi|_{\LIQHP}((\hat{f}\circ R\ten\id)\h{\Gam}'(R(\hat{x})))\\
&=\vphi|_{\LIQHP}(R(\hat{x}))\la\hat{f},1\ra\\
&=\hat{m}(\hat{x})\la\hat{f},1\ra\\
\end{align*}
for all $\hat{x}\in\LIQH$ and $\hat{f}\in\LOQH$. Thus, $\hat{m}$ is a left invariant mean on $\LIQH$.

When $\G$ is co-commutative, the equivalence of $(1)$ and $(3)$ is simply an application of Leptin's theorem \cite{Lep}. When $\h{\G}$ has the approximation property, $(3)\Rightarrow(1)$ was shown in \cite[Corollary 7.4]{C2}.
\end{proof}

\begin{thm}\label{t:coamenequiv} Let $\G$ be a locally compact quantum group. The following conditions are equivalent:
\begin{enumerate}
\item $\G$ is co-amenable;
\item $\LIQ$ is semi-discrete in $\modLOQ_1$;
\item $\LIQ$ is semi-discrete in $\LOQmod_1$;
\item $\mathrm{RUC}(\G)$ is nuclear in $\modLOQ_1$;
\item $\mathrm{LUC}(\G)$ is nuclear in $\LOQmod_1$;
\item The inclusion $C_0(\G)\hookrightarrow C_0(\G)^{**}$ is nuclear in $\modLOQ_1$;
\item The inclusion $C_0(\G)\hookrightarrow C_0(\G)^{**}$ is nuclear in $\LOQmod_1$.
\end{enumerate}
\end{thm}

\begin{proof} $(1)\Rightarrow(2)$ and $(1)\Rightarrow(3)$ follow from Proposition \ref{p:nuc}. 

$(2)\Rightarrow (1)$: Suppose that $\LIQ$ is semi-discrete in $\modLOQ_1$. Then there exist diagrams of normal morphisms in $\modLOQ_1$
\begin{equation*}
\begin{tikzcd}
                     &M_{n_{\alpha}}(\LIQ\oplus_\infty\C)\arrow[dr, "s_\alpha"]\\
\LIQ \arrow[ru, "r_\alpha"] \arrow[rr, "\id"] & &\LIQ,
\end{tikzcd}
\end{equation*}
which approximately commute in the point-weak* topology. Also, by (the right module version of) Proposition \ref{p:semiinj}, $\LIQ$ is injective in $\modLOQ_1$, so that $\h{\G}$ is amenable \cite[Theorem 5.1]{C}. In turn, amenability of $\h{\G}$ implies $\McbQl\cong C_0^u(\G)^*$ canonically and completely isometrically  \cite[Theorem 7.2]{C2}. Combined with the canonical isometric isomorphism $\McbQl\cong\McbQr$ between left and right multipliers (see, e.g., \cite[4.19]{HNR2}), we have $\McbQr\cong C_0^u(\G)^*$ isometrically. Note that
$$r_\alpha\in\mc{CB}^{\sigma}_{\LOQ}(\LIQ,M_{n_{\alpha}}(\LIQ\oplus_\infty\C))=M_{n_{\alpha}}(\mc{CB}^{\sigma}_{\LOQ}(\LIQ,\LIQ\oplus_\infty\C)),$$
so we may write $r_\alpha=[r^\alpha_{i,j}]$, with each $r^\alpha_{i,j}\in \mc{CB}^{\sigma}_{\LOQ}(\LIQ,\LIQ\oplus_\infty\C)$. Let $p_1$ and $p_2$ be the canonical projections from $\LIQ\oplus_\infty\C$ onto its first and second summands. By the isomorphism (\ref{e:Theta}), for each $i,j$ we have
$$p_1\circ r^\alpha_{i,j}\in\mc{CB}^{\sigma}_{\LOQ}(\LIQ,\LIQ)\cong\McbQr\cong C_0^u(\G)^*,$$
(where we used the fact that $p_1$ is a normal morphism) and $p_2\circ r^\alpha_{i,j}\in(\LIQ)_*=\LOQ$, so there exist $\mu_{i,j}^\alpha\in C_0^u(\G)^*$ and $f_{i,j}^\alpha\in\LOQ$ for which
$$r_{i,j}^\alpha(x)=(\Theta^r(\mu_{i,j}^\alpha)(x),\la x,f_{i,j}^\alpha\ra), \ \ \ x\in\LIQ.$$
The $\LOQ$-module property implies that
\begin{align*}(\Theta^r(\mu_{i,j}^\alpha)(x)\star g,\la f_{i,j}^\alpha\star x,g\ra)&=(\Theta^r(\mu_{i,j}^\alpha)(x\star g),\la x\star g,f^\alpha_{i,j}\ra)\\
&=r_{i,j}^\alpha(x\star g)\\
&=r_{i,j}^\alpha(x)\star g\\
&=(\Theta^r(\mu_{i,j}^\alpha)(x),\la x,f_{i,j}^\alpha\ra)\star g\\
&=(\Theta^r(\mu_{i,j}^\alpha)(x)\star g,\la \Theta^r(\mu_{i,j}^\alpha)(x),g\ra)
\end{align*}
for all $g\in\LOQ$. Hence, 
$$\Theta^r(\mu_{i,j}^\alpha)(x)=f_{i,j}^\alpha\star x=\Theta^r(f_{i,j}^\al)(x), \ \ \ x\in\LIQ,$$
which forces $\mu_{i,j}^\alpha=f_{i,j}^\alpha\in \LOQ$ by injectivity of $\Theta^r$. In summary,
$$r_\alpha(x)=[(f_{i,j}^\alpha\star x,\la f_{i,j}^\al,x\ra)],$$
with $f_{i,j}^\alpha\in\LOQ$.

Next, observe that $s_\alpha$ satisfies 
$$s_\alpha([y_{i,j}])=\sum_{i,j=1}^{n_\alpha}s^{\al}_{i,j}(y_{i,j}), \ \ \ [y_{i,j}]\in M_{n_{\alpha}}(\LIQ\oplus_\infty\C),$$
for some collection of morphisms $s^\al_{i,j}:\LIQ\oplus_\infty\C\rightarrow\LIQ$. Since $\LIQ$ is faithful, Lemma \ref{l:nuc} implies that $s^\al_{i,j}=\psi^\al_{i,j}\circ p_1$ for normal morphisms $\psi^\al_{i,j}\in\Hom(\LIQ)$. Thus, $\psi^\al_{i,j}=\Theta^r(\nu_{i,j}^\alpha)$ for some $\nu_{i,j}^\al\in C_0^u(\G)^*$. The composition $s_\al\circ r_\al$ then satisfies
$$s_\al\circ r_\al(x)=\sum_{i,j=1}^{n_\al} \Theta^r(\nu_{i,j}^\alpha)(f_{i,j}^\al\star x)=\sum_{i,j=1}^{n_\al} (\nu_{i,j}^\al\star f_{i,j}^\al)\star x=f_\al\star x, \ \ \ x\in \LIQ$$
where $f_\al=\sum_{i,j=1}^{n_\al}\nu_{i,j}^\al\star f_{i,j}^\al\in \LOQ$ as $\LOQ$ is an ideal in $C_0^u(\G)^*$. Moreover, the isometry $\McbQr\cong C_0^u(\G)^*$ implies
$$\norm{f_{\al}}_{\LOQ}=\norm{\Theta^r(f_\alpha)}_{cb}=\norm{s_\al\circ r_\al}_{cb}\leq 1,$$
for all $\alpha$. Since $f_\al\star x\rightarrow x$ weak* for all $x\in \LIQ$, the standard convexity argument yields a bounded right approximate identity for $\LOQ$, which entails the co-amenability of $\G$.

$(3)\Rightarrow(1)$ follows similarly using left multipliers.

$(1)\Rightarrow(4)$ and $(1)\Rightarrow (5)$ follow from Proposition \ref{p:nuc}.

$(4)\Rightarrow (1)$ is a $C^*$-analogue of $(2)\Rightarrow(1)$, using the $\LOQ$-WEP of $C_0(\G)$ in lieu of $\LOQ$-injectivity of $\LIQ$ to utilize the multiplier representation. If $\RUC$ is nuclear in $\modLOQ_1$ then there exist diagrams of morphisms in $\modLOQ_1$
\begin{equation*}
\begin{tikzcd}
                     &M_{n_{\alpha}}(\LIQ\oplus_\infty\C)\arrow[dr, "s_\alpha"]\\
\RUC \arrow[ru, "r_\alpha"] \arrow[rr, "\id"] & &\RUC,
\end{tikzcd}
\end{equation*}
which approximately commute in the point-norm topology. Since 
$$C_0(\G)\subseteq\RUC\subseteq M(C_0(\G))\subseteq C_0(\G)^{**}$$
(see \cite{Runde}), restricting the morphisms $r_\alpha$ to $C_0(\G)$ shows that the inclusion $C_0(\G)\hookrightarrow C_0(\G)^{**}$ is nuclear. Thus, $C_0(\G)$ has the $\LOQ$-WEP by Proposition \ref{p:nucWEP} and therefore $\h{\G}$ is amenable by Proposition \ref{p:Wepamen}. As above, combining \cite[Theorem 7.2]{C2} with the canonical isometry $\McbQl\cong\McbQr$ we see that $\McbQr\cong C_0^u(\G)^*$ isometrically. Note that
$$r_\alpha|_{C_0(\G)}\in\mc{CB}_{\LOQ}(C_0(\G),M_{n_{\alpha}}(\LIQ\oplus_\infty\C))=M_{n_{\alpha}}(\mc{CB}_{\LOQ}(C_0(\G),\LIQ\oplus_\infty\C)),$$
so we may write $r_\alpha=[r^\alpha_{i,j}]$, with each $r^\alpha_{i,j}\in \mc{CB}_{\LOQ}(C_0(\G),\LIQ\oplus_\infty\C)$. Letting $p_1$ and $p_2$ be the canonical projections from $\LIQ\oplus_\infty\C$ onto its first and second summands, for each $i,j$ it follows (as above) that
$$p_1\circ r^\alpha_{i,j}\in\mc{CB}_{\LOQ}(C_0(\G),\LIQ)\cong\McbQr\cong C_0^u(\G)^*,$$
where the first equality uses \cite[Proposition 4.1]{JNR}. Also, $p_2\circ r^\alpha_{i,j}|_{C_0(\G)}\in C_0(\G)^*=M(\G)$, so there exist $\mu_{i,j}^\alpha\in C_0^u(\G)^*$ and $\nu_{i,j}^\alpha\in M(\G)$ for which
$$r_{i,j}^\alpha(x)=(\Theta^r(\mu_{i,j}^\alpha)(x),\la x,\nu_{i,j}^\alpha\ra), \ \ \ x\in C_0(\G).$$
As in the proof of $(2)\Rightarrow(1)$, the $\LOQ$-module property implies that
$$(\Theta^r(\mu_{i,j}^\alpha)(x)\star g,\la \nu_{i,j}^\alpha\star x,g\ra)
=(\Theta^r(\mu_{i,j}^\alpha)(x)\star g,\la \Theta^r(\mu_{i,j}^\alpha)(x),g\ra)$$
for all $g\in\LOQ$, which forces $\Theta^r(\mu_{i,j}^\alpha)(x)=\nu_{i,j}^\alpha\star x$ for all $x\in C_0(\G)$ and therefore $\mu_{i,j}^\alpha=\nu_{i,j}^\alpha\in M(\G)$ by injectivity of $\Theta^r$. In summary,
$$r_\alpha(x)=[(\nu_{i,j}^\alpha\star x,\la \nu_{i,j}^\al,x\ra)],$$
with $\nu_{i,j}^\alpha\in M(\G)$.

Next, observe that $s_\alpha$ satisfies 
$$s_\alpha([y_{i,j}])=\sum_{i,j=1}^{n_\alpha}s^{\al}_{i,j}(y_{i,j}), \ \ \ [y_{i,j}]\in M_{n_{\alpha}}(\LIQ\oplus_\infty\C),$$
for some collection of morphisms $s^\al_{i,j}:\LIQ\oplus_\infty\C\rightarrow\LIQ$. Since $\LIQ$ is faithful, Lemma \ref{l:nuc} implies that $s^\al_{i,j}=\psi^\al_{i,j}\circ p_1$ for morphisms $\psi^\al_{i,j}\in\Hom(\LIQ)$. Thus, $\psi^\al_{i,j}|_{C_0(\G)}=\Theta^r(\mu_{i,j}^\alpha)$ for some $\mu_{i,j}^\al\in C_0^u(\G)^*$. The composition $s_\al\circ r_\al$ then satisfies
$$s_\al\circ r_\al(x)=\sum_{i,j=1}^{n_\al} \Theta^r(\mu_{i,j}^\alpha)(\nu_{i,j}^\al\star x)=\sum_{i,j=1}^{n_\al} (\mu_{i,j}^\al\star \nu_{i,j}^\al)\star x=\nu_\al\star x, \ \ \ x\in C_0(\G),$$
where $\nu_\al=\sum_{i,j=1}^{n_\al}\mu_{i,j}^\al\star \nu_{i,j}^\al\in M(\G)$. Moreover,
$$\norm{\nu_{\al}}_{M(\G)}=\norm{\Theta^r(\nu_\alpha)}_{cb}=\norm{s_\al\circ r_\al}_{cb}\leq 1,$$
for all $\alpha$. Since $\nu_\al\star x\rightarrow x$ in norm for all $x\in C_0(\G)$, the net $(\nu_\alpha)$ clusters to a right identity for $M(\G)$, which entails the co-amenability of $\G$ (see the end of the proof of \cite[Theorem 5.12]{C} for details).

$(5)\Rightarrow(1)$ is proved similarly as $(4)\Rightarrow (1)$, using left multipliers.

$(1)\Rightarrow(6)$ follows from the implication $(1)\Rightarrow (4)$ and the fact that $(4)\Rightarrow(6)$, as noted above.

$(6)\Rightarrow(1)$: If $C_0(\G)\hookrightarrow C_0(\G)^{**}$ is nuclear in $\modLOQ_1$ then there exist diagrams of morphisms in $\modLOQ_1$
\begin{equation*}
\begin{tikzcd}
                     &M_{n_{\alpha}}(\LIQ\oplus_\infty\C)\arrow[dr, "s_\alpha"]\\
C_0(\G) \arrow[ru, "r_\alpha"] \arrow[rr, "\id"] & &C_0(\G)^{**},
\end{tikzcd}
\end{equation*}
that approximately commute in the point norm topology. Then $C_0(\G)$ has the $\LOQ$-WEP by Proposition \ref{p:nucWEP} and therefore $\h{\G}$ is amenable by Proposition \ref{p:Wepamen}. Composing $s_\alpha$ with the canonical morphism $\pi:C_0(\G)^{**}\twoheadrightarrow\LIQ$, the resulting diagrams
\begin{equation*}
\begin{tikzcd}
                     &M_{n_{\alpha}}(\LIQ\oplus_\infty\C)\arrow[dr, "\pi\circ s_\alpha"]\\
C_0(\G) \arrow[ru, "r_\alpha"] \arrow[rr, "\id"] & &\LIQ,
\end{tikzcd}
\end{equation*}
approximately commute in the point norm topology. An argument similar to the proof of $(4)\Rightarrow (1)$ establishes the co-amenability of $\G$.

Finally, the equivalence $(1)\Leftrightarrow(7)$ is similar to $(1)\Leftrightarrow(6)$.

\end{proof}

\begin{remark} Since $\h{\G}$ is amenable if and only if $\LIQ$ is injective in $\modLOQ_1$ \cite[Theorem 5.1]{C}, Theorem \ref{t:coamenequiv} shows that the long-standing open problem on the equivalence between co-amenability of a locally compact quantum group $\G$ and amenability of $\h{\G}$ \cite{Voi} is the problem of equivalence between injectivity and semi-discreteness of $\LIQ$ in $\mathbf{mod}\hskip2pt\LOQ_1$.
\end{remark}

\begin{remark} It is natural to wonder whether $\LOQ$-nuclearity of $C_0(\G)$ is related to co-amenability of $\G$. When $\G$ is compact, meaning $C_0(\G)$ is unital, this is indeed the case as $\LUC=\RUC=C_0(\G)$. 

If $\G$ is not compact and $C_0(\G)$ is $\LOQ$-nuclear, then $\h{\G}$ is amenable (by Proposition \ref{p:Wepamen}) so by \cite[Proposition 5.5]{C} any $\vphi\in\Hom(\LIQ,\LIQ)$ decomposes into a sum of completely positive morphisms. However, there is no non-zero completely positive $\LOQ$-morphism from $\LIQ$ into $C_0(\G)$, since the image of $1$ in $C_0(\G)$ would necessarily be a fixed point for the $\LOQ$-action on $\LIQ$, forcing it to lie in $\C1$. It follows from Lemma \ref{l:nuc} that there is no non-zero morphism from any $M_n(\LIQ\oplus_\infty \C)$ into $C_0(\G)$, which contradicts the $\LOQ$-nuclearity of $C_0(\G)$. Thus, $C_0(\G)$ is not $\LOQ$-nuclear for non-compact $\G$.
\end{remark}

\begin{remark} There are notions of relative nuclearity/semi-discreteness for inclusions of $C^*$-/von Neumann algebras using the language of correspondences \cite{AD,Has,Popa}. These notions may be viewed as analogues of nuclearity and semi-discreteness in categories of Hilbert modules over the respective algebras. For instance, if $B\subseteq A$ is a unital inclusion of $C^*$-algebras with conditional expectation $A\rightarrow B$, then the pair $(B,A)$ is strongly relatively nuclear in the sense of \cite{Has} if, roughly speaking, the identity map on $A$ can be approximated by certain finite sums of $B$-module maps which factor through $M_n(B)$ in a manner which utilizes the canonical Hilbert $B$-module structure on $A$ induced from the conditional expectation. Viewing $A\in B\hskip2pt\mathbf{mod}$, this module notion of nuclearity is different from that in Definition \ref{d:nuc}, which takes place in a different category, and in which the identity approximately factors through $M_n(B^*)$. For a concrete distinction, the pair $(A,A)$ is always strongly relatively nuclear for any unital $C^*$-algebra $A$. If a $C^*$-algebra $A$ is nuclear in $\Amod_1$ then it necessarily has Lance's weak expectation property by Proposition \ref{p:nucWEP} and \cite[Corollary 3.10]{BC2}. Similar remarks apply to relative semi-discreteness in the sense of \cite{AD}.
\end{remark}

\section{Equivalence of injectivity and semi-discreteness for crossed products}\label{s:dyn}

In this section we establish the equivalence between $A(G)$-injectivity of $G\bar{\ltimes} M$, $A(G)$-semi-discreteness of $G\bar{\ltimes} M$, and amenability of $W^*$-dynamical systems $(M,G,\alpha)$ with $M$ injective. The proof relies on a recent Herz-Schur multiplier characterization of amenable actions \cite[Theorem 3.13]{BC} together with a generalized construction from (the proof of) \cite[Theorem 3.5]{CT} to produce a specific $A(G)$-module left inverse to the dual co-action which lies in the point weak* closure of \textit{normal $A(G)$-module maps}. We begin with the necessary tools from dynamical systems. 


A $W^*$-dynamical system $(M,G,\alpha)$ consists of a von Neumann algebra $M$ endowed with a homomorphism $\alpha:G\rightarrow\mathrm{Aut}(M)$ of a locally compact group $G$ such that for each $x\in M$, the map $G\ni s\rightarrow \alpha_s(x)\in M$ is weak* continuous. We let $M_c$ denote the unital $C^*$-subalgebra consisting of those $x\in M$ for which $s\mapsto \alpha_s(x)$ is norm continuous. By \cite[Lemma 7.5.1]{Ped}, $M_c$ is weak* dense in $M$.

Every $W^*$-dynamical system induces a normal $G$-equivariant injective $*$-homomorphism $\alpha:M\rightarrow\LI\oten M$ via
$$\alpha(x)(s)=\alpha_{s^{-1}}(x), \ \ \ x\in M, \ s\in G,$$
and a corresponding right $\LO$-module structure on $M$ \cite[18.6]{Str}. The crossed product of $M$ by $G$, denoted $G\bar{\ltimes}M$, is the von Neumann subalgebra of $\BLT\oten M$ generated by $\alpha(M)$ and $VN(G)\ten 1$. 

The system $(M,G,\alpha)$ admits a dual co-action 
$$\wh{\alpha}:G\bar{\ltimes}M\rightarrow VN(G)\oten (G\bar{\ltimes}M)$$
of $VN(G)$ on the crossed product, given by
\begin{equation}\label{e:coaction}\h{\alpha}(X)=(\h{W}^*\ten 1)(1\ten X)(\h{W}\ten 1), \ \ \ X\in G\bar{\ltimes}M,\end{equation}
where $\h{W}$ is the left fundamental unitary of $VN(G)$. On the generators we have $\h{\alpha}(\h{x}\ten 1)=(\h{W}^*(1\ten\hat{x})\h{W})\ten 1$, $\hat{x}\in VN(G)$ and $\h{\alpha}(\alpha(x))=1\ten\alpha(x)$, $x\in M$. This co-action yields a canonical right operator $A(G)$-module structure on the crossed product via
$$X\cdot u=(u\ten\id)\h{\alpha}(X), \ \ \ X\in G\bar{\ltimes}M, \ u\in A(G).$$

A $C^*$-dynamical system $(A,G,\alpha)$ consists of a $C^*$-algebra endowed with a homomorphism $\alpha:G\rightarrow\mathrm{Aut}( A)$ of a locally compact group $G$ such that for each $a\in A$, the map $G\ni s\mapsto\alpha_s(a)\in A$ is norm continuous. 

A covariant representation $(\pi, \sigma)$ of $( A,G,\alpha)$ consists of a representation $\pi: A\rightarrow\BH$ and a unitary representation $\sigma:G\rightarrow\BH$ such that $\pi(\alpha_s(a))=\sigma_s\pi(a)\sigma_{s^{-1}}$ for all $a \in A$, $s\in G$. Given a covariant representation $(\pi,\sigma)$, we let
$$(\pi \times \sigma)(f) = \int_G \pi(f(t)) \sigma_t \, dt, \ \ \ f\in C_c(G, A).$$
The full crossed product $G\ltimes_f A$ is the completion of $C_c(G, A)$ in the norm
$$\|f \| = \sup_{(\pi, \sigma)} \| (\pi \times \sigma)(f)\|$$
where the $\sup$ is taken over  all covariant representations $(\pi, \sigma)$ of $(A,G,\alpha)$.

Let $A\subseteq\mc{B}(H)$ be a faithful non-degenerate representation of $ A$. Then $(\alpha,\lm\ten 1)$ is a covariant representation on $L^2(G,H)$, where 
$$\alpha(a)\xi(t)=\alpha_{t^{-1}}(a)\xi(t), \ \ \ (\lm\ten 1)(s)\xi(t)=\xi(s^{-1}t), \ \ \ \xi\in L^2(G,H).$$
The reduced crossed product $G\ltimes A$ is defined to be the norm closure of $(\alpha\times(\lm\ten 1))(C_c(G,A))$. This definition is independent of the faithful non-degenerate representation $A\subseteq\mc{B}(H)$. We often abbreviate $\alpha\times(\lm\ten 1)$ as $\alpha\times\lm$. 

Analogous to the group setting, dual spaces of crossed products can be identified with certain $ A^*$-valued functions on $G$. We review aspects of this theory below and refer the reader to \cite[Chapters 7.6, 7.7]{Ped} for details. 

For each $C^*$-dynamical system $( A,G,\alpha)$ there is a universal covariant representation $(\pi,\sigma)$ such that 
$$G\ltimes_f A\subseteq C^*(\pi( A)\cup \sigma(G))\subseteq M(G\ltimes_f A).$$
Each functional $\vphi\in (G\ltimes_f A)^*$ then defines a function $u:G\rightarrow A^*$ by
\begin{equation} \label{eqn: FS relation} \la u(s),a\ra=\vphi(\pi(a)\sigma_s), \ \ \ a\in A, \ s\in G.\end{equation}
Let $B(G\ltimes_f A)$ denote the resulting space of $ A^*$-valued functions on $G$. An element $u\in B(G\ltimes_f A)$ is \textit{positive definite} if it arises from a positive linear functional $\vphi$ as above. We let $A(G\ltimes_f A)$ denote the subspace of $B(G\ltimes_f A)$ whose associated functionals $\vphi$ are of the form
$$\vphi(x)=\sum_{n=1}^\infty\la\xi_n, (\alpha\times\lm)(x)\eta_n\ra, \ \ \ x\in G\ltimes_f  A,$$
for sequences $(\xi_n)$ and $(\eta_n)$ in $L^2(G,H)$ with $\sum_{n=1}^\infty \norm{\xi_n}^2<\infty$ and $\sum_{n=1}^\infty\norm{\eta_n}^2<\infty$. Then $A(G\ltimes_f A)$ is a norm closed subspace of $(G\ltimes_f A)^*$ which can be identified with $((G\ltimes A)'')_*$.

Let $(A,G,\alpha)$ be a $C^*$-dynamical system. We let $L^2(G,A)$ be the right Hilbert $A$-module given by the completion of $C_c(G,A)$ under $\norm{\xi}=\norm{\la\xi,\xi\ra}_A^{1/2}$, where
$$\la\xi,\zeta\ra=\int_G\xi(s)^*\zeta(s) \ ds, \ \ \ \xi\cdot a(s)=\xi(s)a, \ \ \ \xi,\zeta\in C_c(G,A), \ a\in A.$$
We let $\lm_s\ten\alpha_t\in \mc{B}(L^2(G,A))$ denote the isometry
$$(\lm_s\ten\alpha_t)\xi(r)=\alpha_t(\xi(s^{-1}r)), \ \ \ \xi\in C_c(G,A), \ s,t\in G.$$
By left invariance of the Haar measure and continuity of the action it follows that 
$$\la(\lm_s\ten\alpha_t)\xi,(\lm_s\ten\alpha_t)\zeta\ra=\alpha_t(\la\xi,\zeta\ra), \ \ \ \xi,\zeta\in L^2(G,A), \ s,t\in G.$$

\begin{prop}\label{p:posdef} Let $(M,G,\alpha)$ be a $W^*$-dynamical system, and let $\xi\in C_c(G,Z(M)_c)$. The function $h:G\times G\ni(s,t)\mapsto \la\xi,(\lm_s\ten\alpha_t)\xi\ra$ defines a normal completely positive $A(G)$-module map $\Phi_h:VN(G)\oten( G\bar{\ltimes} M)\rightarrow G\bar{\ltimes} M$ satisfying $\norm{\Phi_h}_{cb}=\norm{\la\xi,\xi\ra}$ and 
$$(\Phi_h)_*:A(G\ltimes_f M_c)\ni u\mapsto (1\ten u)\cdot h\in A((G\times G)\ltimes_f (\C\ten M_c)).$$
\end{prop}

\begin{proof} We first construct a map at the level of the Fourier--Stieltjes space $B(G\ltimes_f M_c)$. To that end, fix $u\in B(G\ltimes_f M_c)^+$. We show that $(1\ten u)\cdot h\in B((G\times G)\ltimes_f (\C\ten M_c))^+$, where $(\C\ten M_c,G\times G,\tr\ten\alpha)$ is the tensor product system with the trivial action on $\C$, and the operation $\cdot$ is the pointwise action of $M_c$ on its dual. Given $(s_1,t_1),...,(s_n,t_n)\in G\times G$ and $1\ten x_1,...,1\ten x_n$ in $\C\ten M_c$, we have
\begin{align*}&\sum_{j,k=1}^n\la ((1\ten u)\cdot h)(s_j^{-1}s_k,t_j^{-1}t_k),(\tr\ten\alpha)_{(s_j^{-1},t_j^{-1})}(1\ten x_j^*x_k)\ra\\
&=\sum_{j,k=1}^n\la  u(t_j^{-1}t_k)\cdot h(s_j^{-1}s_k,t_j^{-1}t_k),\alpha_{t_j^{-1}}(x_j^*x_k)\ra\\
&=\sum_{j,k=1}^n\la  u(t_j^{-1}t_k),h(s_j^{-1}s_k,t_j^{-1}t_k)\alpha_{t_j^{-1}}(x_j^*x_k)\ra\\
&=\sum_{j,k=1}^n\la  u(t_j^{-1}t_k),\la\xi,\lm_{s_j^{-1}s_k}\ten\alpha_{t_j^{-1}t_k}\xi\ra\alpha_{t_j^{-1}}(x_j^*x_k)\ra\\
&=\sum_{j,k=1}^n\la  u(t_j^{-1}t_k),\alpha_{t_j^{-1}}(\la\lm_{s_j}\ten\alpha_{t_j}\xi,\lm_{s_k}\ten\alpha_{t_k}\xi\ra x_j^*x_k)\ra.
\end{align*}
Since $\xi$ takes values in $Z(M)_c$, it follows that
$$\la\lm_{s_j}\ten\alpha_{t_j}\xi,\lm_{s_k}\ten\alpha_{t_k}\xi\ra x_j^*x_k=\int_G(\lm_{s_j}\ten\alpha_{t_j}\xi(r))^*x_j^*x_k(\lm_{s_k}\ten\alpha_{t_k}\xi)(r) \ dr.$$
Hence,
\begin{align*}&\sum_{j,k=1}^n\la ((1\ten u)\cdot h)(s_j^{-1}s_k,t_j^{-1}t_k),(\tr\ten\alpha)_{(s_j^{-1},t_j^{-1})}(1\ten x_j^*x_k)\ra\\
&=\int_G\sum_{j,k=1}^n\la  u(t_j^{-1}t_k),\alpha_{t_j^{-1}}((\lm_{s_j}\ten\alpha_{t_j}\xi(r))^*x_j^*x_k(\lm_{s_k}\ten\alpha_{t_k}\xi)(r))\ra\\
&=\int_G\sum_{j,k=1}^n\vphi_u\bigg(\alpha_{t_j^{-1}}((\lm_{s_j}\ten\alpha_{t_j}\xi(r))^*x_j^*x_k(\lm_{s_k}\ten\alpha_{t_k}\xi)(r))\sigma(t_j^{-1})\sigma(t_k)\bigg)\\
&=\int_G\vphi_u\bigg(\bigg(\sum_{j=1}^kx_j(\lm_{s_j}\ten\alpha_{t_j}\xi(r))\sigma(t_j)\bigg)^*\bigg(\sum_{k=1}^kx_k(\lm_{s_k}\ten\alpha_{t_k}\xi(r))\sigma(t_k)\bigg)\bigg)\\
&\geq 0.
\end{align*}
It follows from \cite[Proposition 7.6.8]{Ped} (applied to the $C^*$-dynamical system $(\C\ten M_c,G\times G,\tr\ten\alpha)$) that $(1\ten  u)\cdot h\in B((G\times G)\ltimes_f (\C\ten M_c))^+$. In particular, we obtain a well defined linear map 
$$h:B(G\ltimes_f M_c)\ni u\mapsto (1\ten  u)\cdot h\in B((G\times G)\ltimes_f (\C\ten M_c))$$
through the Jordan decomposition. Since $(M_n(\C)\ten M_c,G,\id_{M_n}\ten\alpha)$ and $(M_n(\C)\ten \C\ten M_c,G\times G,\id_{M_n}\ten\tr\ten\alpha)$ are $C^*$-dynamical systems satisfying $M_n(\C)\ten(G\ltimes_f M_c)\cong G\ltimes_f(M_n(\C)\ten M_c)$ and 
$$M_n(\C)\ten((G\times G)\ltimes_f (\C\ten M_c))\cong (G\times G)\ltimes_f(M_n(\C)\ten \C\ten M_c)$$
canonically (by \cite[Lemma 2.75]{W}), and since \cite[Proposition 7.6.8]{Ped} applies to any $C^*$-dynamical system, the matricial analogue of the above argument together with the previous identifications show that the linear map
$$h:B(G\ltimes_f M_c)\ni u\mapsto (1\ten  u)\cdot h\in B((G\times G)\ltimes_f (\C\ten M_c))$$
is completely positive. Moreover, since the left marginal of $h$ is compactly supported, if $ u$ is compactly supported, then so is $(1\ten  u)\cdot h$. Since compactly supported elements of $B((G\times G)\ltimes_f (\C\ten M_c))^+$ lie in $A((G\times G)\ltimes_f (\C\ten M_c))^+$ \cite[Lemma 7.7.6]{Ped}, it follows that $h$ induces a completely positive map
$$h:A(G\ltimes_f M_c)\ni  u\mapsto (1\ten  u)\cdot h\in A((G\times G)\ltimes_f (\C\ten M_c)).$$
Since $A((G\times G)\ltimes_f (\C\ten M_c))^*\cong ((G\times G)\ltimes(\C\ten M_c))^{''}\cong VN(G)\oten (G\bar{\ltimes} M)$, we obtain a completely positive $A(G)$-module map
$$\Phi_h:=h^*:VN(G)\oten (G\bar{\ltimes} M)\rightarrow  G\bar{\ltimes} M,$$
where the module structure on the domain is on the left leg. Moreover,
$$\la\Phi_h(1),u\ra=\la1\ten1,(1\ten u)\cdot h\ra=\la1,u(e)h(e,e)\ra=\la h(e,e),u(e)\ra=\la\alpha(h(e,e)),u\ra$$
for all $u\in A(G\ltimes_f M_c)$, so it follows that
$$\norm{\Phi_h}_{cb}=\norm{\Phi_h(1)}=\norm{\alpha(h(e,e))}=\norm{h(e,e)}=\norm{\la\xi,\xi\ra}.$$


\end{proof}

\begin{thm} Let $(M,G,\alpha)$ be a $W^*$-dynamical system with $M$ injective (as a von Neumann algebra). Then the following are equivalent.
\begin{enumerate}
\item $(M,G,\alpha)$ is amenable;
\item $G\bar{\ltimes} M$ is $A(G)$-semi-discrete;
\item $G\bar{\ltimes} M$ is $A(G)$-injective.
\end{enumerate}
\end{thm}

\begin{proof} $(1)\Rightarrow (2)$: By the proof of \cite[Theorem 3.13]{BC}, there exists a net $(\xi_i)\in C_c(G,Z(M)_c)$ whose corresponding Herz-Schur multipliers $h_i(s)=\la\xi_i,(\lm_s\ten\alpha_s)\xi_i\ra$ satisfy $h_i(e)=\la\xi_i,\xi_i\ra=1$ for all $i$ and $S_i\rightarrow\id_{G\bar{\ltimes}M}$ point weak*, where $S_i$ are the normal completely positive $A(G)$-module maps on $G\bar{\ltimes}M$ determined by
$$S_i((\alpha\times\lm)(f))=\int_G \alpha(h(s)f(s))(\lm_s\ten 1) \ ds, \ \ \ f\in C_c(G,M_c).$$
(Note the slightly different notation from \cite{BC}.)

Let $h_i$ also denote the function
$$G\times G\ni(s,t)\mapsto \la\xi_i,(\lm_s\ten\alpha_t)\xi_i\ra\in Z(M)_c,$$
and let $\Phi_{h_i}:VN(G)\oten( G\bar{\ltimes} M)\rightarrow G\bar{\ltimes} M$ be the associated completely positive $A(G)$-module map from Proposition \ref{p:posdef}. Since $\norm{\la\xi_i,\xi_i\ra}\leq 1$, each $\Phi_{h_i}$ is completely contractive. We claim that $\Phi_{h_i}\circ\wh{\alpha}=S_i$, where $\wh{\alpha}:G\bar{\ltimes}M\rightarrow VN(G)\oten( G\bar{\ltimes} M)$ is the dual co-action. To this end, let $f\in C_c(G,M_c)$. Then for any $u\in A(G\ltimes_f M_c)$, we have
\begin{align*}\la \Phi_{h_i}(\wh{\alpha}((\alpha\times\lm)(f))),u\ra&=\la \wh{\alpha}((\alpha\times\lm)(f)),((1\ten u)\cdot h_i)\ra\\
&=\int_G\la (1\ten \alpha(f(s)))(\lm_s\ten\lm_s\ten 1) \ ds,(1\ten u)\cdot h_i\ra\\
&=\int_G\la f(s),u(s)\cdot h_i(s,s)\ra \ ds\\
&=\int_G\la h_i(s)f(s),u(s)\ra \ ds\\
&=\int_G\la \alpha(h_i(s)f(s))(\lm_s\ten 1),u\ra \ ds\\
&=\la S_i((\alpha\times\lm)(f)),u\ra.
\end{align*}
Normality then establishes the claim.

Now, since $M$ is injective and $(M,G,\alpha)$ is amenable, $G\bar{\ltimes} M$ is an injective von Neumann algebra \cite[Proposition 3.12]{ADI}. Pick diagrams
\begin{equation*}
\begin{tikzcd}
                     &M_{n_j} \arrow[dr, "s_j"]\\
G\bar{\ltimes} M \arrow[ru, "r_j"] \arrow[rr, "\id"] & &G\bar{\ltimes} M
\end{tikzcd}
\end{equation*}
which approximately commute in the point-weak* topology. We then obtain diagrams
\begin{equation*}
\begin{tikzcd}
(VN(G)\oplus\C1)\oten (G\bar{\ltimes} M) \arrow[r, "\id\ten r_j"]&M_{n_j}(VN(G)\oplus\C1) \arrow[r, "\id\ten s_j"] &(VN(G)\oplus\C1)\oten (G\bar{\ltimes} M)\arrow[d, "\Phi_{h_i}\circ(p_1\ten \id)"]\\
G\bar{\ltimes} M \arrow[u, "\Delta_+"] \arrow[rr, "\id"] & &G\bar{\ltimes} M
\end{tikzcd}
\end{equation*}
of morphisms in $\mathbf{mod}\hskip2pt A(G)_1$. Since $(p_1\ten\id)\circ\Delta_+=\wh{\alpha}$ (as is easily verified), $\Phi_{h_i}\circ\wh{\alpha}=S_i$ for each $i$, and $S_i\rightarrow\id_{G\bar{\ltimes}M}$ point weak*, it follows that the diagrams approximately commute in the point-weak* topology after the appropriate iterated limit (first in $j$ then in $i$). Hence, $G\bar{\ltimes} M$ is $A(G)$-semi-discrete.

$(2)\Rightarrow(3)$ follows immediately from Proposition \ref{p:semiinj}.

$(3)\Rightarrow (1)$: By \cite[Theorem 5.2]{BC2}, $A(G)$-injectivity of $G\bar{\ltimes}M$ together with injectivity of $M$ implies that $(M,G,\alpha)$ is amenable.
\end{proof}

\section{Co-exact modules}\label{s:coexact}

Any finite-dimensional normed space $E$ is isomorphic to a subspace of some $\ell^\infty_n$, as well as a quotient of some $\ell^1_m$. Moreover, for every $\ep>0$, there exists a subspace $S$ of some $\ell^\infty_n$ and a quotient $Q$ of $\ell^1_m$ such that $d(E,S)<1+\ep$ and $d(E,Q)<1+\ep$, where $d$ is the Banach--Mazur distance. The operator space analogues of these properties are false, and their failure is measured by the notions of \textit{exactness} \cite{Pisier} and \textit{co-exactness} \cite[\S4.5]{Webster}. 

A finite-dimensional operator space $E$ is \textit{$\lm$-exact} if for every $\ep>0$, there exists $n\in\N$ and a subspace $S\subseteq M_n$ such that $d_{cb}(E,S)<\lm+\ep$, where $d_{cb}$ is the completely bounded Banach--Mazur distance. For example, $\max \ell^1_3$ and $T_3$ are not 1-exact. Dually, a finite-dimensional operator space $E$ is \textit{$\lm$-co-exact} if for every $\ep>0$, there exists $n\in\N$ and a quotient $Q$ of $T_n$ such that $d_{cb}(E,Q)<\lm+\ep$. Clearly, $E$ is $\lm$-co-exact if and only if $E^*$ is $\lm$-exact. It is known that $\lm$-co-exactness is equivalent to the $\lm$-OLLP of \cite{Ozawa} for finite-dimensional operator spaces (see \cite[Theorem 2.5]{Ozawa}).

Pisier showed that a finite-dimensional operator space $E$ is $\lm$-exact if and only if for any family $(X_i)_{i\in I}$ of operator spaces and any free ultrafilter $\mc{U}$ on $I$, 
$$E\iten \prod_{i\in I} X_i/\mc{U}\cong_\lm\prod_{i\in I} (E\iten X_i)/\mc{U}$$
\cite[Proposition 6]{Pisier}. Dually, it was shown by Dong that a finite-dimensional operator space $E$ is 1-co-exact if and only if for any family $(X_i)_{i\in I}$ of operator spaces and any free ultrafilter $\mc{U}$ on $I$, 
\begin{equation}\label{e:Dong}E\pten \prod_{i\in I} X_i/\mc{U}\cong\prod_{i\in I} (E\pten X_i)/\mc{U}\end{equation}
completely isometrically \cite[Theorem 2.2]{Dong2}. The goal of this section is to show that a similar phenomena to (\ref{e:Dong}) persists at the level of operator modules. We remark that Dong's argument from \cite{Dong2}, which factors through Pisier's characterization of exactness and results from \cite{Ozawa}, does not readily generalize to the module setting.

\begin{defn} Let $A$ be a completely contractive Banach algebra, $E\in\Amod$ and $\lm\geq 1$. $E$ is \textit{$\lm$-co-exact} if for every $\ep>0$ there exists $n\in\N$ and a complete quotient $Q$ of $T_n(A_+)$ with $d_{cb}(Q,E)<\lm+\ep$. 
\end{defn}

\begin{remark} There is an analogous notion of $\lm$-exact operator modules using submodules of finitely co-generated co-free modules of the form $M_n(A_+^*)$, $n\in\N$. A detailed investigation of this (and related) notion(s) will appear in forthcoming work. See the outlook section for more details.
\end{remark}

In what follows we adopt the notation from \cite[\S10.3]{ER} surrounding ultraproducts. By \cite[Proposition 10.3.2]{ER} it follows that for any family $(X_i)_{i\in I}$ in $\modA$, and any free ultrafilter $\mc{U}$ on $I$, the ultraproduct $\prod_{i\in I} X_i/\mc{U}\in\modA$ via 
$$\pi_{\mc{U}}((x_i))\cdot a=\pi_{\mc{U}}((x_i\cdot a)), \ \ \ (x_i)\in\prod_{i\in I} X_i, \ a\in A.$$
Given $Y\in\Amod$, the canonical map
$$\bigg(\prod_{i\in I} X_i/\mc{U}\bigg)\pten_A Y\ni \pi_{\mc{U}}((x_i))\ten_A y\mapsto\pi_{\mc{U}}(x_i\ten_A y)\in \prod_{i\in I} (X_i\pten_A Y)/\mc{U}$$
is completely contractive by the universal property of $\pten_A$. 

In preparation for the lemma below, note that any finite-dimensional operator space $F$ is a complete quotient of $T_\infty:=T_{\N}$ (see notation from section \ref{ss:op}). Indeed, as $F$ is separable there exists a Banach space quotient map $\ell^1\quo F$. Equipping $\ell^1$ with its max operator space structure, this map becomes a complete quotient map $\max\ell^1\quo F$. Composing this with the canonical conditional expectation $T_\infty\quo\max\ell^1$ gives the desired mapping.

\begin{lem}\label{l:diso} Let $F$ be a $d$-dimensional operator space, and let $q:T_\infty\quo F$ be a complete quotient map. For any $0<\ep<1/2$, there exists $n_0\in\N$ such that for all $n\geq n_0$,
$$d_{cb}(F,T_n/\mathrm{Ker}(q_n))\leq d^2\frac{(1+\ep)}{(1-\ep)},$$
where $q_n=q|_{T_n}:T_n\rightarrow F$.
\end{lem}

\begin{proof} Let $\{x_1,...,x_d\}$ be a normalized Auerbach basis of $F$ with dual basis $\{x_1^*,...,x_d^*\}$, meaning $\la x_i^*,x_j\ra=\delta_{i,j}$, $1\leq i,j\leq d$, and $\norm{x_i}=\norm{x_i^*}=1$, $1\leq i\leq d$. Given $0<\ep<1/2$, pick $y_1,...,y_d\in (T_\infty)_{\norm{\cdot}<1+\ep}$ for which $q(y_k)=x_k$, $k=1,..,d$. There exists $n_0\in\N$ such that $\norm{P_ny_kP_n-y_k}<\ep/d$ for all $n\geq n_0$, where $P_n(\cdot)P_n$ is the natural compression onto the $n^{th}$ block. Define $p_n:F\rightarrow T_n$ by $p_n(x_k)=P_ny_kP_n$. Then for any $x=\sum_{k=1}^d \la x_k^*,x\ra x_k\in F$, we have
$$\norm{p_n(x)}\leq\sum_{k=1}^d|\la x_k^*,x\ra|\norm{y_k}\leq(1+\ep)d\norm{x},$$
implying $\norm{p_n}\leq d(1+\ep)$. Moreover, 
$$\norm{q(p_n(x))-x}\leq\sum_{k=1}^d|\la x_k^*,x\ra|\norm{q(p_n(x_k))-x_k}\leq\norm{x}\sum_{k=1}^d\norm{P_ny_kP_n-y_k}<\ep\norm{x},$$
Then, as in the proof of \cite[Proposition 5.3]{ER3}, since $\ep<1/2$, 
$$\norm{x}-\norm{q(p_n(x))}\leq\norm{x-q(p_n(x))}<\frac{1}{2}\norm{x} \ \ \Rightarrow \ \ \norm{q(p_n(x))}\geq\frac{1}{2}\norm{x}.$$
Hence, $q\circ p_n:F\rightarrow F$ is injective and therefore bijective by finite-dimensionality of $F$. It follows that $q_n=q|_{T_n}:T_n\rightarrow F$ is surjective, and therefore $\widetilde{q_n}:T_n/\mathrm{Ker}(q_n)\simeq F$. Moreover, for each $x\in F_{\norm{\cdot}=1}$ there exists a unique $y_x\in F$ satisfying $\widetilde{q_n}^{-1}(x)=p_n(y_x)+\mathrm{Ker}(q_n)$. Then
$$\norm{y_x}-1=\norm{y_x}-\norm{x}\leq\norm{y_x-x}=\norm{y_x-q_n(p_n(y_x))}<\ep\norm{y_x},$$
so that $\norm{y_x}<1/(1-\ep)$. Thus,
$$\norm{\widetilde{q_n}^{-1}(x)}=\norm{p_n(y_x)+\mathrm{Ker}(q_n)}\leq\norm{p_n(y_x)}\leq d\frac{(1+\ep)}{(1-\ep)},$$
implying $\norm{\widetilde{q_n}^{-1}}\leq d(1+\ep)/(1-\ep)$. By automatic complete boundedness in finite dimensions \cite[Corollary 2.2.4]{ER} we have $\norm{\widetilde{q_n}^{-1}}_{cb}\leq d^2(1+\ep)/(1-\ep)$.
\end{proof}

Let $I$ be a set and $\mc{U}$ be a free ultrafilter on $I$. It is well-known that the ultrapower functor $\Ban_1\ni X\mapsto X^{\mc{U}}\in\Ban_1$ is exact \cite[Lemma 2.2.g]{CG}. We require a generalization of this fact in the operator module setting. Recall that a morphism $\vphi:X\rightarrow Y$ is \textit{$\lm$-strict} if $\vphi(X)$ is closed and the induced map $\widetilde{\vphi}:X/\mathrm{Ker}(\vphi)\rightarrow \vphi(X)$ satisfies $\norm{\widetilde{\vphi}^{-1}}_{cb}\leq\lm$.

\begin{lem}\label{l:ultra} Let $A$ be a completely contractive Banach algebra. Let $I$ be a set and $\mc{U}$ be a free ultrafilter on $I$. Suppose that for each $i\in I$, the sequence 
\begin{equation*}
\begin{tikzcd}
X_i \arrow[r, "\psi_i"] &Y_i \arrow[r, "\vphi_i"] &Z_i,
\end{tikzcd}
\end{equation*} 
is exact in $\Amod$ with $(\psi_i)$, and $(\vphi_i)$ uniformly completely bounded, $\psi_i$ a $\lm$-strict morphism and $\vphi_i$ a $\mu$-strict morphism for each $i$. Then the ultraproduct sequence
\begin{equation*}
\begin{tikzcd}
\prod_{i\in I}X_i/\mc{U} \arrow[r, "(\psi_i)_\mc{U} "] &\prod_{i\in I}Y_i/\mc{U} \arrow[r, "(\vphi_i)_\mc{U} "] &\prod_{i\in I}Z_i/\mc{U},
\end{tikzcd}
\end{equation*} 
is exact.
\end{lem}

\begin{proof} 
If $\pi_\mc{U}(y_i)\in\mathrm{Ker}((\vphi_i)_\mc{U})$, then $\lim_{i\mapsto\mc{U}}\norm{\vphi_i(y_i)}=0$, so for each $\delta>0$, there is some $S\in\mc{U}$ such that
$$\norm{\vphi_{i}(y_{i})}<\frac{\delta}{\mu}, \ \ \ i\in S.$$
Since each $\vphi_i$ is a $\mu$-strict morphism
$$\norm{y_{i}+\mathrm{Ker}(\vphi_{i})}\leq\mu\norm{\vphi_{i}(y_{i})}<\delta, \ \ \ \ i\in S.$$
Pick $k_{i}\in\mathrm{Ker}(\vphi_{i})$ for which $\norm{y_{i}-k_{i}}<\delta$. Since 
$$\mathrm{Ker}(\vphi_{i})=\mathrm{Im}(\psi_{i})\cong_\lm X_{i}/\mathrm{Ker}(\psi_{i}),$$
there exists $x_{i}\in X_{i}$ for which $\psi_{i}(x_{i})=k_{i}$ and
$$\norm{x_{i}+\mathrm{Ker}(\psi_{i})}\leq\lm\norm{k_{i}}<\lm(\delta+\sup_i\norm{y_i}), \ \ \ i\in S.$$
Thus, setting $x_i=x_{i}+h_{i}$ for suitable elements $h_{i}\in\mathrm{Ker}(\psi_{i})$, and setting $x_i=0$ whenever $i\notin S$, we obtain a bounded family $(x_i)\in\prod_i X_i$ for which
$$\norm{(\psi_i)_{\mc{U}}(\pi_{\mc{U}}((x_i)))-\pi_{\mc{U}}((y_i))}<\delta.$$
Since $\delta>0$ was arbitrary, it follows that $\pi_{\mc{U}}((y_i))\in\mathrm{Im}((\psi_i)_\mc{U})$.
\end{proof}

\begin{defn}\label{d:fg} Let $A$ be a completely contractive Banach algebra. A module $E\in\Amod$ is \textit{finitely generated} if there exists a complete quotient morphism $A_+\pten F\twoheadrightarrow E$ for some finite-dimensional operator space $F$.
\end{defn}

\begin{remark} A module $M$ over a unital ring $R$ is finitely generated if there is an epimorphism $R^m\twoheadrightarrow M$ from a finitely generated free $R$-module. 

The proof of Proposition \ref{p:refp} shows that for any finitely generated $E\in\Amod$ (in the sense of Definition \ref{d:fg}), there exists $\lm\geq 1$, $n\in\N$ and a complete $\lm$-quotient morphism $T_n(A_+)\twoheadrightarrow_\lm E$. Thus, any finitely generated module in $\Amod$ is (cb isomorphic to) a quotient of a finitely generated matricially free module. 
\end{remark}

\begin{thm}\label{t:ul} Let $A$ be a completely contractive Banach algebra, $E\in\Amod$ be finitely generated, and $\lm\geq1$. Consider the following conditions: 
\begin{enumerate}
\item for any family $(X_i)_{i\in I}$ in $\modA$, and any free ultrafilter $\mc{U}$ on $I$, the canonical map
$$\Phi_E:\bigg(\prod_{i\in I} X_i/\mc{U}\bigg)\pten_A E\rightarrow\prod_{i\in I} (X_i\pten_A E)/\mc{U}$$
is a complete isomorphism with $\norm{\Phi_E^{-1}}_{cb}\leq\lm$;
\item for any family $(X_i)_{i\in I}$ in $\Amod$, and any free ultrafilter $\mc{U}$ on $I$, the canonical map
$$\Delta:\prod_{i\in I} \Hom(E,X_i)/\mc{U}\rightarrow \Hom(E,\prod_{i\in I} X_i/\mc{U})$$
is a complete $\lm$-embedding;
\item $E\in\Amod$ is $\lm$-co-exact.
\end{enumerate}
Then $(1)\Rightarrow(2)\Rightarrow (3)$. If, in addition, $E$ is projective, the conditions are equivalent.
\end{thm}

\begin{proof} $(1)\Rightarrow(2)$: Let $I$ be a set, $(X_i)$ be a family in $\Amod$ and $\mc{U}$ be a free ultrafilter on $I$. By assumption, the canonical map 
$$\Phi_E:\bigg(\prod_{i\in I} X_i^*/\mc{U}\bigg)\pten_A E\rightarrow\prod_{i\in I}(X_i^*\pten_A E)/\mc{U}$$
is a complete isomorphism with $\norm{\Phi_E^{-1}}_{cb}\leq\lm$. Let 
$$T:\Hom\bigg(E,\prod_{i\in I}X_i/\mc{U}\bigg)\rightarrow\bigg(\prod_{i\in I}(X_i^*\pten_A E)/\mc{U}\bigg)^*$$
denote the following composition
\begin{align*}\Hom\bigg(E,\prod_{i\in I}X_i/\mc{U}\bigg)&\hookrightarrow\Hom\bigg(E,\prod_{i\in I}X_i^{**}/\mc{U}\bigg)\\
&\hookrightarrow\Hom\bigg(E,\bigg(\prod_{i\in I}X_i^*/\mc{U}\bigg)^*\bigg)\\
&\cong\bigg(\bigg(\prod_{i\in I}X_i^*/\mc{U}\bigg)\pten_A E\bigg)^*\\
&\xrightarrow{(\Phi_E^{-1})^*}\bigg(\prod_{i\in I}(X_i^*\pten_A E)/\mc{U}\bigg)^*.
\end{align*}
Then $\norm{T}_{cb}\leq\lm$.

By \cite[Proposition 10.3.2]{ER} and \cite[Corollary 10.3.4]{ER}, the composition
$$I:\prod_{i\in I}\Hom(E,X_i)/\mc{U}\hookrightarrow \prod_{i\in I}\Hom(E,X_i^{**})/\mc{U}= \prod_{i\in I}(X_i^*\pten_A E)^*/\mc{U}\hookrightarrow \bigg(\prod_{i\in I}(X_i^*\pten_A E)/\mc{U}\bigg)^*$$
is a complete isometry. The following diagram is easily seen to commute
\begin{equation*}
\begin{tikzcd}
\prod_{i\in I}\Hom(E,X_i)/\mc{U}\arrow[r, hook, "I"]\arrow[d, "\Delta"] & \bigg(\prod_{i\in I}(X_i^*\pten_A E)/\mc{U}\bigg)^*\arrow[d, equal]\\
\Hom\bigg(E,\prod_{i\in I}X_i/\mc{U}\bigg)\arrow[r, "T"] & \bigg(\prod_{i\in I}(X_i^*\pten_A E)/\mc{U}\bigg)^*,
\end{tikzcd}
\end{equation*}
where $\Delta$ denotes (right) composition with the canonical embedding $\Delta:E\hookrightarrow E^{\mc{U}}=\prod_{i\in I}E/\mc{U}$ (see \cite[Lemma 10.3.1]{ER}). Then for any $[(\vphi^{k,l}_i)_\mc{U}]\in M_n\bigg(\prod_{i\in I}\Hom(E,X_i)/\mc{U}\bigg)$, we have
$$\norm{[(\vphi^{k,l}_i)_\mc{U}]}=\norm{[I((\vphi^{k,l}_i)_\mc{U})]}=\norm{[T(\Delta((\vphi^{k,l}_i)_\mc{U}))]}\leq\lm\norm{[\Delta((\vphi^{k,l}_i)_\mc{U})]}.$$
Hence, $\Delta$ is a complete $\lm$-embedding.

$(2)\Rightarrow(3)$: Since $E$ finitely generated there is a finite-dimensional operator space $F$ and a complete quotient morphism $\vphi:A_+\pten F\twoheadrightarrow E$. Pick a complete quotient map $q:T_\infty\quo F$, and let $q_n:=q|_{T_n}:T_n\rightarrow F$. Fix some $0<\ep<1/2$. By Lemma \ref{l:diso}, there exists $n_0\in\N$ such that $\norm{\widetilde{q_n}^{-1}}_{cb}\leq d^2(1+\ep)/(1-\ep)\leq 3d^2$ for all $n\geq n_0$. 

Let $K_n=\mathrm{Ker}(\vphi\circ(\id_{A_+}\ten q_n))$ and $Q_n=(A_+\pten T_n)/K_n$. Since $K_n=(\id_{A_+}\ten q_n)^{-1}(\mathrm{Ker}(\vphi))$, the standard argument shows that
$$\vphi\circ(\id_{A_+}\ten q_n):A_+\pten T_n\rightarrow E$$ 
induces a complete isomorphism $\vphi_n:Q_n\cong E$ with $\norm{\vphi_n^{-1}}_{cb}\leq 3d^2$ for all $n\geq n_0$. 

Put $\vphi_n=0$ (and $\vphi_n^{-1}=0$) for all $n<n_0$, and let $\mc{U}$ be a free ultrafilter on $\N$. Since the $\vphi_n$ have uniformly completely bounded inverses asymptotically,
$$(\vphi_n)_{\mc{U}}:\prod_{n\in\N}Q_n/\mc{U}\cong\prod_{n\in\N}E/\mc{U}$$ 
is a complete isomorphism. By $(2)$, we have
$$\lim_{n\mapsto\mc{U}}\norm{\vphi_n^{-1}}_{cb}=\norm{\pi_{\mc{U}}((\vphi_n^{-1}))}\leq\lm\norm{\Delta((\vphi_n)^{-1}_\mc{U})}_{cb}=\lm\norm{(\vphi_n)_{\mc{U}}^{-1}\circ\Delta}_{cb}.$$ 
However,
$$(\vphi_n)_{\mc{U}}^{-1}\circ\Delta:E\rightarrow\prod_{n\in\N}Q_n/\mc{U}$$
is a contraction: given $x\in E_{\norm{\cdot}<1}$, pick $y\in (A_+\pten T_\infty)_{\norm{\cdot}<1}$ for which
$$x=\vphi((\id_{A_+}\ten q)(y))=\lim_n\vphi((\id_{A_+}\ten q_n)(\Ad(1_{A_+}\ten P_n)(y)))=\lim_n\vphi_n((\Ad(1_{A_+}\ten P_n)(y)).$$
Then $\pi_{\mc{U}}((\Ad(1_{A_+}\ten P_n)(y)+K_n))$ lies in the unit ball of $\prod_{n\in\N}Q_n/\mc{U}$ and 
$$\Delta(x)=(\vphi_n)_{\mc{U}}(\pi_{\mc{U}}((\Ad(1_{A_+}\ten P_n)(y)+K_n))).$$
It follows that $(\vphi_n)_{\mc{U}}^{-1}(\Delta(x))=\pi_{\mc{U}}((\Ad(1_{A_+}\ten P_n)(y)+K_n))$ has norm less than 1. A similar argument shows that $(\vphi_n)_{\mc{U}}^{-1}\circ\Delta$ is completely contractive. Thus, for every $\ep>0$
$$\lim_{n\mapsto\mc{U}}\norm{\vphi_n^{-1}}_{cb}<\lm+\ep,$$
so there exists $n\in\N$ for which $\norm{\vphi_n^{-1}}_{cb}<\lm+\ep$, and $E$ is $\lm$-co-exact.

$(3)\Rightarrow(1)$: Now, suppose that $E$ is $\lm$-co-exact and projective. Then by an argument similar to Example \ref{ex:fp2}, for every $\ep>0$ there is an exact sequence in $\Amod$
\begin{equation*}
\begin{tikzcd}
T_n(A_+) \arrow[r, "\psi"] &T_n(A_+) \arrow[r, two heads, "\vphi"] &E,
\end{tikzcd}
\end{equation*}
for which $\psi$ is a completely bounded projection onto $\mathrm{Ker}(\vphi)$ and $\vphi$ is a complete $(\lm+\ep)$-quotient map. 

Observe that for any $n\in\N$,
\begin{align*}\bigg(\prod_{i\in I} X_i/\mc{U}\bigg)\pten_A (A_+\pten T_n)&=\bigg(\bigg(\prod_{i\in I} X_i/\mc{U}\bigg)\pten_A A_+\bigg)\pten T_n\\
&\cong\bigg(\prod_{i\in I} X_i/\mc{U}\bigg)\pten T_n\\
&\cong\prod_{i\in I} (X_i\pten T_n)/\mc{U} \ \ \ \ \textnormal{(by \cite[Lemma 10.3.8]{ER})}\\
&\cong\prod_{i\in I} (X_i\pten_A (A_+\pten T_n))/\mc{U}.
\end{align*}
We therefore obtain the commutative diagram
\begin{equation}\label{d:diag}
\begin{tikzcd}
& \bigg(\prod_{i\in I} X_i/\mc{U}\bigg)\pten_A T_n(A_+)\arrow[r, "\id\ten\psi"]\arrow[d, equal] & \bigg(\prod_{i\in I} X_i/\mc{U}\bigg)\pten_A T_n(A_+)\arrow[r, two heads, "\id\ten\vphi"]\arrow[d, equal] & \bigg(\prod_{i\in I} X_i/\mc{U}\bigg)\pten_A E\arrow[d, "\Phi"]\\
& \prod_{i\in I} (X_i\pten_A T_n(A_+))/\mc{U}\arrow[r] & \prod_{i\in I} (X_i\pten_A T_n(A_+))/\mc{U}\arrow[r,two heads] &\prod_{i\in I} (X_i\pten_A E)/\mc{U}.
\end{tikzcd}
\end{equation}
Since $(\id_{X_i}\ten\vphi)$ is a $(\lm+\ep)$-complete quotient map for each $i$, it follows from the proof of \cite[Proposition 10.3.2]{ER} that the second arrow in the bottom row is also a $(\lm+\ep)$-complete quotient map. Hence, so too is $\Phi$. Showing that $\Phi$ is injective will prove that
$$d_{cb}\bigg(\bigg(\prod_{i\in I} X_i/\mc{U}\bigg)\pten_A E,\prod_{i\in I} (X_i\pten_A E)/\mc{U}\bigg)\leq\lm+\ep.$$
Since $\ep>0$ was arbitrary, the claim would follow. To that end, Lemma \ref{l:rightexact} implies the sequence
\begin{equation*}
\begin{tikzcd}
X_i\pten_A T_n(A_+) \arrow[r, "\id_{X_i}\ten_A\psi"] &X_i\pten_A T_n(A_+) \arrow[r, two heads, "\id\ten_{X_i}\ten_A\vphi"] & X_i\pten_A E,
\end{tikzcd}
\end{equation*}
is topologically exact for each $i$. Moreover, as $\psi$ is a completely bounded projection, the sequence is exact as $(\id_{X_i}\ten_A\psi)$ is a completely bounded projection onto its closed range. Moreover, as a completely bounded projection,  $(\id_{X_i}\ten_A\psi)$ is a 1-strict morphism for each $i$. It then follows from Lemma \ref{l:ultra} that the bottom row of (\ref{d:diag}) is exact. Since the top row is topologically exact (by Lemma \ref{l:rightexact}),  the same diagram chasing from the proof of Proposition \ref{p:fp} entails the injectivity of $\Phi$. 

\end{proof}

\begin{remark} Inspection of the proof shows that $(3)\Rightarrow(1)$ is valid for any finitely presented module $E\in\Amod$ for which there is an exact sequence
\begin{equation*}
\begin{tikzcd}
T_m(A_+) \arrow[r, "\psi"] &T_n(A_+) \arrow[r, two heads, "\vphi"] &E,
\end{tikzcd}
\end{equation*}
in $\Amod$ with $\vphi$ a complete $\lm$-quotient map and $\psi$ ``stably $\mu$-strict'' in the sense that the amplified morphism $\id_X\ten_A\psi:X\pten_A T_m(A_+)\rightarrow X\pten_A T_n(A_+)$ is $\mu$-strict for any $X\in\modA$. Projectivity is a natural assumption that ensures this feature.
\end{remark}

\section{Outlook}

Several natural lines of investigation are suggested by this work. First and foremost, the dual notion of topological finite co-presentation, its relation to operator module analogues of (weak*) exactness, and connection to exact quantum groups will appear in forthcoming work. 

With the dual notions of finite (co)-presentation at hand, which are suitable analogues of finite dimensionality, one can formulate operator module analogues of more general finite dimensional approximation properties and pursue a Grothendieck type programme in the category of operator modules.

\section*{Acknowledgements}

The original idea which led to this project was formulated during the Fields Institute Thematic Program on Abstract Harmonic Analysis, Banach and Operator Algebras in 2014. The author is grateful for the institute for its hospitality during the program. The author would like to thank Yemon Choi, Matthias Neufang and Nico Spronk for helpful conversations at various points throughout the project. We are also grateful to the anonymous referee whose suggestions improved the presentation of the paper. This work was partially supported by the NSERC Discovery Grant RGPIN-2017-06275.

\end{spacing}

\end{document}